\documentclass{amsart}

\usepackage{amsmath}
\usepackage{amssymb}
\usepackage{mathrsfs}
\usepackage{url}
\usepackage{enumitem}
\usepackage{stmaryrd}
\usepackage[all]{xy}
\usepackage{verbatim}

\newtheorem{Thm}{Theorem}[section]
\newtheorem{Con}[Thm]{Conjecture}
\newtheorem{Cor}[Thm]{Corollary}
\newtheorem{Def}[Thm]{Definition}

\newtheorem{Lem}[Thm]{Lemma}
\newtheorem{Prop}[Thm]{Proposition}
 
\newtheorem{Rmk}[Thm]{Remark}
\newtheorem{Prob}[Thm]{Problem}
\newtheorem{Question}[Thm]{Question}
\numberwithin{equation}{Thm}

\newcommand{\FF}{\mathbb{F}}

\newcommand{\RR}{\mathbb{R}}

\newcommand{\ZZ}{\mathbb{Z}}

\DeclareMathOperator{\Sing}{Sing}

\date{\today}

\author{Daqing Wan}
\address{Center for Discrete Mathematics, Chongqing University, Chongqing 401331, China.}
\address{College of Math. and Statistics, Chongqing University, Chongqing 401331, China.}
\email{dwan@math.uci.edu}

\title{NP-hardness of SVP in Euclidean space}

\begin{document}

\begin{abstract}
In 1981, van Emde Boas conjectured that computing a shortest non-zero vector of a lattice in a Euclidean space is $\mathbf{NP}$-hard. In this paper, we prove this conjecture, thereby derandomizing Ajtai's classical randomized hardness result (1998). We follow the derandomization program formulated by Micciancio (1998--2014) 
who conjectured the existence of an efficient deterministic construction of locally dense lattices. The key is to resolve this conjecture. 
Our proof builds on the candidate construction via Reed-Solomon codes by Bennett and Peikert (2023), and depends crucially on Deligne's work on the Weil conjectures for higher-dimensional varieties over finite fields.

\end{abstract}

\maketitle

%\section{Introduction}

\section{Introduction} 

A lattice $\mathcal{L}$ in the $n$-dimensional Euclidean space $\RR^n$ is the set of all integer linear combinations of $m$ linearly independent column vectors ${\bf v}_1, \dots, {\bf v}_m$ in $\ZZ^n$, where $m\leq n$ is called the dimension of the lattice $\mathcal{L}$. Note that we only consider integral lattices in $\RR^n$. The matrix $B = ({\bf v}_1, \dots, {\bf v}_m)$ is called a basis of $\mathcal{L}$. The lattice $\mathcal{L}$ generated by $B$ is defined as 
\[ \mathcal{L} = \mathcal{L}(B) := \left\{ \sum_{i=1}^m a_i {\bf v}_i : a_1, \dots, a_m \in \ZZ \right\}. \]

Lattices are classically studied mathematical objects that originated from the pioneering work of Gauss \cite{G1801} over two centuries ago. They have proved to be invaluable in numerous computer science applications, including integer programming \cite{Len83, Kan87}, coding theory \cite{CS99}, cryptanalysis \cite{LO85, Cop01, Bon99}, and, most notably, lattice-based cryptography \cite{Ajt96, MG02, Pei16}, which offers the critical advantage of resisting known efficient quantum attacks. Indeed, the security of lattice-based cryptography fundamentally relies on the hardness of computational lattice problems, making this a highly active research area in both theoretical computer science and cryptography. 

The central algorithmic problem in lattice theory is the Shortest Vector Problem (SVP): given a lattice basis $B$ as input, find a shortest non-zero vector in $\mathcal{L}(B)$. The fundamental question here is 
\begin{Question}
Is there a polynomial-time algorithm to solve $\mathrm{SVP}$?
\end{Question}
The celebrated LLL algorithm \cite{LLL82} and its subsequent improvements find an approximate short lattice vector that may be exponentially longer than the true shortest non-zero vector. Currently, all algorithms solving exact $\mathrm{SVP}$ have exponential worst-case running time.

From its introduction, SVP was widely believed to be computationally intractable for large dimensions $n$. Nevertheless, its precise complexity has remained a long-standing open problem in theoretical computer science; see \cite{Ben23, Zon25, Men26} for recent surveys on this subject and its foundational applications to post-quantum cryptography. 

\begin{Con}[van Emde Boas \cite{vEB81}]\label{Con1.0} 
{\rm SVP} is {\bf NP}-hard. 
\end{Con}

This conjecture implies ${\rm SVP}$ is at least as hard as the most difficult problems in ${\bf NP}$, 
showing that ${\rm SVP}$ has no polynomial time algorithm under the standard assumption that ${\bf NP} \neq {\bf P}$, thereby answering the fundamental algorithmic question regarding SVP. 

Almost twenty years after van Emde Boas's pioneering work, a major breakthrough was made by Ajtai \cite{Ajt98}, who proved that SVP is {\bf NP}-hard under randomized polynomial-time reduction, by reducing the {\bf NP}-complete subset sum problem to ${\rm SVP}$. This implies that 
${\rm SVP}$ has no polynomial time algorithm under the stronger assumption that ${\bf NP} \neq {\bf RP}$.  

To prove van Emde Boas's conjecture in full, one must derandomize Ajtai's randomized polynomial-time reduction. Numerous authors have attempted this over the years. In particular, Micciancio \cite{Mic98} simplified Ajtai's sophisticated construction and formalized a derandomization approach via the lifting of 
a linear code and by reducing the {\bf NP}-hard Closest Vector Problem (CVP) to ${\rm SVP}$. This  assumes the existence of an efficient, deterministic construction of a certain combinatorial gadget: a locally dense lattice with an explicit bad center (see \cite[Conjecture 1]{Mic14}). Micciancio proposed two strategies to construct this gadget efficiently. The first \cite{Mic98} relies on the conjectural distribution of smooth numbers in very short intervals; however, this number-theoretic conjecture remains far out of reach, even assuming the Generalized Riemann Hypothesis. The second \cite{Mic12} attempts to construct the gadget using the lifting of a tower of binary BCH codes, which has also proven difficult to fully realize. 

More recently, Bennett and Peikert \cite{BP23} proposed a third approach, constructing the gadget by lifting Reed-Solomon codes over large prime finite fields. This method appears simpler and more amenable to derandomization, as its coding-theoretic analogue \cite{DMS99}—the construction of a similar gadget using Reed-Solomon codes over large finite fields—was successfully derandomized over a decade ago by Cheng and Wan \cite{CW09} using Weil's bound for character sums on curves (see also \cite{KA11, Mic14} for subsequent simpler proofs). However, as demonstrated in \cite{BP23}, a direct application of the Weil bound fails to yield useful information in the lattice setting. Indeed, derandomization in the lattice setting is far more profound, requiring techniques that go well beyond the Weil bound. In this paper, we show that this barrier can be overcome by appealing to some of the deepest theorems in arithmetic geometry, specifically utilizing the full strength of Deligne's theorem on the Weil conjectures. Our main result is:

\begin{Thm}\label{Thm1.1} 
$\mathrm{SVP}$ is {\bf NP}-hard. 
\end{Thm}

This settles Conjecture \ref{Con1.0}. The Shortest Vector Problem can be naturally generalized to the $\ell_p$ norm for any $1 \leq p \leq \infty$, rather than being restricted to the Euclidean norm ($p=2$). Furthermore, while the aforementioned SVP is formulated as a search problem, to prove its hardness, it suffices to establish the hardness of its decision version. Following standard literature, we investigate the more general $\gamma$-approximate decision problem in the $\ell_p$ norm, denoted by $\gamma\text{-}\mathrm{GapSVP}_p$, where $1 \leq p \leq \infty$, and the approximation factor satisfies $\gamma \geq 1$ (which may be a constant or a function of the dimension $n$). Recall that the $\ell_p$ norm in $\RR^n$ is defined by 
\[ \|(x_1, \dots, x_n)\|_p = \left(\sum_{i=1}^n |x_i|^p\right)^{1/p}, \quad \|(x_1, \dots, x_n)\|_{\infty} = \max_{1 \leq i \leq n} |x_i|. \]

The case $p=2$ corresponds to the Euclidean norm, which is of paramount importance in applications, especially in lattice-based cryptography. One motivation for studying the $p \neq 2$ case is that the problem structurally tends to become more tractable as $p$ increases. Indeed, for $p=\infty$, van Emde Boas \cite{vEB81} established that SVP is ${\bf NP}$-hard in the $\ell_{\infty}$ norm. Additionally, one hopes that techniques developed to resolve the large $p$ regime can eventually be adapted to handle smaller values of $p$. 

We now formalize the approximation version and its $\ell_p$-norm analogue for $1 \leq p \leq \infty$. A fundamental geometric parameter of any lattice is its minimum distance. 

\begin{Def} 
The $\ell_p$-norm minimum distance of a lattice $\mathcal{L}$ in $\RR^n$ is 
\[ \lambda^{(p)}(\mathcal{L}) = \min_{v \in \mathcal{L} \setminus \{0\}} \|v\|_p. \]
\end{Def}

The $\gamma$-approximate decision version is formulated as the following promise problem, where the approximation factor $\gamma$ is assumed to satisfy $\gamma \geq 1$. 

\begin{Prob}[$\gamma\text{-}\mathrm{GapSVP}_p$] Fix $1\leq p < \infty$ and constant $\gamma \geq 1$. 
Given a lattice basis $M$ and a distance threshold $d>0$, decide whether $\lambda^{(p)}(\mathcal{L}) \leq d$ or $\lambda^{(p)}(\mathcal{L}) > \gamma d$ when one of the two cases is promised to hold. 
\end{Prob}

For the exact problem ($\gamma = 1$), we write $\mathrm{GapSVP}_p$. Let $\mathrm{SVP}_p$ denote the search problem of computing a shortest non-zero lattice vector in the $\ell_p$ norm; thus, $\mathrm{SVP}$ is simply $\mathrm{SVP}_2$. If $\mathrm{GapSVP}_p$ is ${\bf NP}$-hard, then $\mathrm{SVP}_p$ is also ${\bf NP}$-hard. Since van Emde Boas \cite{vEB81} proved that $\mathrm{GapSVP}_{\infty}$ is ${\bf NP}$-hard, it is natural to expect that his conjecture extends from $p=2$ to all finite $p \geq 1$. 

\begin{Con}\label{Con2.0} 
For every $1\leq p < \infty$, $\mathrm{GapSVP}_p$ is {\bf NP}-hard. 
\end{Con}

This generalized conjecture has been studied extensively. As mentioned earlier, the first major breakthrough was established by Ajtai for the classical Euclidean case $p=2$. 

\begin{Thm}[Ajtai \cite{Ajt98}] 
$\mathrm{GapSVP}_2$ is {\bf NP}-hard under randomized polynomial-time reductions. 
\end{Thm}

Ajtai's hardness result for the exact problem ($\gamma=1$) was subsequently extended to increasingly larger constant approximation factors $\gamma \geq 1$ (see, e.g., \cite{CN98, Mic98, Kho03, Kho04, RR06, HR07, Mic12}). It is known that for all $p \geq 1$, $\gamma\text{-}\mathrm{GapSVP}_p$ is ${\bf NP}$-hard for any constant $\gamma \geq 1$ under randomized reductions. It remains hard even for nearly polynomial factors $\gamma = n^{\Omega(1/\log\log n)}$, under randomized subexponential-time reductions. 

Note that for $p \geq 2$, $\gamma\text{-}\mathrm{GapSVP}_p$ is unlikely to be ${\bf NP}$-hard for large factors $\gamma \geq c_p\sqrt{n}$, where $c_p > 0$ is a constant depending only on $p$ (see \cite{GG98, AR04, Pei07}). Conversely, the security of lattice-based cryptography often relies on the conjectural hardness of $\gamma\text{-}\mathrm{GapSVP}_p$ for much larger approximation factors $\gamma$, typically polynomial in $n$. This discrepancy suggests that proving standard ${\bf NP}$-hardness is unlikely for the parameter regimes of interest in cryptography. In particular, finding a short non-zero lattice vector $v$ satisfying $\|v\|_2 \leq \sqrt{m} |\det(B^tB)|^{1/2m}$ (known as Minkowski-SVP) is expected to be computationally hard, but likely not ${\bf NP}$-hard. 

Derandomizing these hardness results has been a long-standing open problem. Historically, the primary framework for derandomization was the construction of locally dense lattices via coding theory as formulated by Micciancio. More recently, alternative approaches based on probabilistically checkable proofs (PCPs) and subset-sum variants have also proved highly fruitful. For $p > 2$, powerful deterministic hardness results have been established in several recent works. Most notably, Hair and Sahai \cite{HS25*} showed that for any $p>2$ and any $\epsilon>0$, $\gamma\text{-}\mathrm{GapSVP}_p$ is ${\bf NP}$-hard for $1 \leq \gamma < 2^{\log^{1-\epsilon}n}$. This establishes the deterministic ${\bf NP}$-hardness of ${\rm SVP}_p$ for all $p>2$, settling Conjecture \ref{Con2.0} in the $p>2$ regime. However, the PCP-based approach of \cite{HS25*} currently cannot prove Conjecture \ref{Con2.0} in the  $1 \leq p \leq 2$ regime, which includes the practically critical Euclidean case $p=2$. 

In Hecht and Safra \cite[Theorem 1.1]{HS25}, also utilizing PCP techniques, it is shown that for any constant $p>2$ and constant $1 \leq \gamma < \sqrt{2}$, $\gamma\text{-}\mathrm{GapSVP}_p \notin {\bf P}$ unless $\mathrm{3SAT} \in \mathbf{DTIME}(2^{O(n^{2/3}\log n)})$. While this provides a deterministic hardness result, it relies on a sub-exponential time reduction (rather than a polynomial-time reduction) and remains restricted to $p>2$. Hittmeir \cite{Hit25} proved a fine-grained deterministic hardness result for $\gamma\text{-}\mathrm{GapSVP}_p$ parameterized by $p$, via a reduction to a variant of the subset-sum problem. Crucially, none of these recent techniques have succeeded in establishing deterministic ${\bf NP}$-hardness for $1 \leq p \leq 2$, most notably leaving the $p=2$ case open. 

In a recent preprint, Hair and Sahai \cite{HS26**} independently achieved a breakthrough by proving a deterministic hardness result for all $1 \leq p < \infty$ via the PCP framework. Their reduction is sub-exponential in time and thus does not settle the deterministic ${\bf NP}$-hardness for the $1 \leq p \leq 2$ regime. 

We now focus on the development of locally dense lattices, which is the framework adopted in this work. Classically, derandomizing Ajtai's result (for $p=2$) has been approached by reduction to the Closest Vector Problem (CVP), which has long been known to be ${\bf NP}$-hard \cite{vEB81}. In \cite{Mic98, Mic12}, this derandomization is reduced to constructing a specific combinatorial gadget. Roughly speaking, this gadget consists of a locally dense lattice $\mathcal{L}$ and an explicit coset $s + \mathcal{L}$ containing subexponentially many short vectors whose norms are at most a fraction of the lattice's minimum distance. Consequently, list decoding is infeasible for the ``bad center" $s$ at this error radius. While constructing such a locally dense lattice $\mathcal{L}$ is generally not very difficult, and a random choice of $s$ yields a bad center with high probability (giving a randomized polynomial-time reduction), derandomization requires an explicit, deterministic construction of the bad center $s$. Although the bad center can easily be guessed, the fundamental difficulty lies in proving that this candidate indeed satisfies the necessary properties. This is where a rigorous mathematical proof becomes necessary. 

In \cite{Mic98}, Micciancio proposed a number-theoretic approach to explicitly construct $\mathcal{L}$ and $s$. However, this framework relies on the unproven conjecture that for any $\epsilon > 0$, there exists $d > 0$ such that the short interval $[n, n + n^{\epsilon}]$ contains an odd, square-free, $\log^d n$-smooth integer for all sufficiently large $n$. This number-theoretic conjecture is currently far out of reach, even assuming the Generalized Riemann Hypothesis. In \cite{Mic12}, a more sophisticated method was proposed based on a tower of binary BCH codes. This construction has the advantage of yielding a randomized reduction with only one-sided error, unlike the two-sided error reductions in \cite{Kho03, HR07}. Thus, the framework in \cite{Mic12} partially derandomized these reductions and suggested a viable path toward full derandomization. However, despite the elegance of this coding-theoretic connection, the binary field $\FF_2$ is too small to permit the application of powerful tools from arithmetic geometry. 

A few years ago, Bennett and Peikert \cite{BP23} introduced a third approach to construct the desired lattice gadget by lifting Reed-Solomon codes over large prime finite fields, which notably provided a much simpler proof of Ajtai's randomized hardness result. This Reed-Solomon lattice framework, though originally randomized, appeared highly amenable to derandomization. Indeed, the analogous gadget in coding theory—known as a locally dense code—was deterministically constructed by Cheng and Wan \cite{CW09} using Weil's bound for character sums. That coding-theoretic result established the deterministic ${\bf NP}$-hardness of approximating the minimum distance of a linear code, thereby derandomizing the main result of \cite{DMS99}. However, derandomizing the lattice analogue in \cite{BP23} requires proving that the proposed bad center of the Reed-Solomon lattice (which is a lifting of a Reed-Solomon code, rather than a code itself) is indeed valid. This task is significantly more challenging: as discussed in \cite{BP23}, the classical Weil bound is insufficient to yield any non-trivial bounds in the lattice setting. 

In this work, we demonstrate that the much deeper Deligne bound for higher-dimensional varieties can be successfully applied to prove that the proposed bad center is indeed valid and satisfies the required projection properties. This resolves Micciancio's open conjecture \cite[Conjecture 1]{Mic14} on the efficient deterministic construction of locally dense lattices. As a direct consequence, we obtain the following:

\begin{Thm}\label{Thm3} 
For every $1\leq p < \infty$, $\gamma\text{-}\mathrm{GapSVP}_p$ is {\bf NP}-hard for all $1\leq \gamma < 2^{1/p}$. 
\end{Thm}

\begin{Cor}
For every $1\leq p < \infty$, $\gamma\text{-}\mathrm{GapSVP}_p \notin {\bf P}$ for all $1\leq \gamma < 2^{1/p}$, assuming ${\bf NP} \neq {\bf P}$. 
\end{Cor}

This theorem fully derandomizes the main result of \cite{BP23}. In the special case $\gamma=1$, it settles van Emde Boas's conjecture (Conjecture \ref{Con1.0}). A naturally arising open problem is to find an elementary proof of this result that bypasses heavy algebraic geometry. Another challenge is to extend this result to larger constant approximation factors $\gamma \geq 1$, potentially via tensor product techniques to amplify the hardness factor. For the $p > 2$ regime, this is already achieved in \cite{HS25*} for sub-polynomial factors via PCPs. It remains a fascinating question whether the PCP techniques of \cite{HS25*, HS25} or the subset-sum approach of \cite{Hit25} can be adapted to the small $p$ regime (particularly $p=2$) to prove the exact {\bf NP}-hardness, perhaps in tandem with Deligne's theorem. 

\begin{Rmk} 
As shown in \cite[Lemma 5.1]{BP23}, any explicit bad center of a locally dense Reed-Solomon lattice yields an explicit list-decoding configuration for Reed-Solomon codes where the agreement-to-dimension ratio exceeds the state-of-the-art bound of Guruswami and Rudra \cite[Corollary 2]{GR05} (which requires a ratio of $2 - \Omega(1)$ to produce super-polynomial list sizes). As an application of our main construction, this agreement-to-dimension ratio can be made arbitrarily large (see Corollary \ref{ratio}). 
\end{Rmk}

The remainder of this paper is organized as follows. In Section 2, we review the formal definition of locally dense lattices and their connection to the derandomization of SVP, following \cite{Mic12, BP23}. In Section 3, we present the definition of Reed-Solomon lattices and their utility in constructing locally dense gadgets, following \cite{BP23}. This reduces the geometric task of constructing locally dense lattices to the study of $\FF_q$-rational points on certain algebraic varieties over a large prime field $\FF_q$. In Section 4, we analyze the geometric properties of these varieties, showing that they are complete intersections with mild singularities. In Section 5, we apply Deligne's theorem (the Riemann Hypothesis for varieties over finite fields) and the total Betti number bounds of \cite{WZ26} to obtain sharp estimates for the number of their $\FF_q$-rational points. In Section 6, we apply the inclusion-exclusion sieve to establish that the Reed-Solomon lattice is locally dense in a weaker sense (without coordinate projection). Finally, in Section 7, we extend this sieving framework to prove local density in the strong sense (with projection), thereby completing the derandomization proof.

\subsection*{Acknowledgements} 
I would like to thank Qi Cheng and Hendrik Lenstra for helpful discussions on these and related topics over the years. My interest in lattices was initiated during a visit to the Institute for Advanced Study at Tsinghua University many years ago; I am grateful to Xiaoyun Wang for her warm hospitality and for hosting my visit to her research group.

\section{Locally dense lattices}

Locally dense lattices, as formulated by Micciancio, are fundamental in proving complexity results 
and in de-randomization. They are also of independent interest as basic objects of study in mathematics and 
theoretical computer science. In this section, we review the notion of locally dense lattices, 
beginning with a description of a simplified version. 

Roughly speaking, a lattice $\mathcal{L}$ in $\mathbb{R}^n$ is called locally dense if for some $\alpha \in (0, 1)$, 
there is $x \in \mathbb{Z}^n$ such that the shift $x + \mathcal{L}$ contains at least $2^{n^{\epsilon}}$ 
elements of $\ell_p$-norm at most $\alpha \lambda^{(p)}(\mathcal{L})$ for some $\epsilon > 0$. That is, the ball 
$B(0, \alpha \lambda^{(p)}(\mathcal{L}))$ centered at $0$ with $\ell_p$-radius $\alpha \lambda^{(p)}(\mathcal{L}) < \lambda^{(p)}(\mathcal{L})$ 
contains subexponentially many vectors in $x+\mathcal{L}$. Namely, 
$$|B(0, \alpha \lambda^{(p)}(\mathcal{L})) \cap (x + \mathcal{L})| \geq 2^{n^{\epsilon}}.$$
Clearly, we must have $\alpha \geq 1/2$ for the ball to contain more than one 
vector in $x+\mathcal{L}$. Furthermore, we require $\alpha > 2^{-1/p}$ for the ball to contain more than a polynomial number of vectors in $x+\mathcal{L}$, at least in the $\ell_2$-norm case; 
see Micciancio and Goldwasser~\cite[Chapter 5]{MG02}. Thus, we can and will assume that $\alpha \in (2^{-1/p}, 1)$. The center $x$ is called a \emph{bad center} of the lattice $\mathcal{L}$. 

Locally dense lattices are not efficiently list-decodable, even combinatorially, to within the 
distance $\alpha \lambda^{(p)}(\mathcal{L})$ around the bad center $x$. 
A locally dense lattice usually has many bad centers, which can often be proved by an averaging argument. 
A suitable sampling would then find a bad center with non-trivial probability, leading to 
hardness results under random polynomial-time reductions. To obtain a deterministic reduction, one needs to deterministically construct a candidate center and mathematically prove that this suspected center is indeed a bad center. This is the key to de-randomization. 

We now recall the formal definition of locally dense lattices following~\cite{Mic12} and~\cite{BP23}. 
The formal definition is slightly more involved than the intuitive description above.

\begin{Def}[Locally dense lattice]\label{def:locally-dense} 
For $1 \leq p < \infty$, a real number $\alpha \in (2^{-1/p}, 1)$, and positive integers 
$r < n$, a $(p, \alpha, r, n)$-locally dense lattice consists of an integer lattice $\mathcal{L}$ in $\mathbb{R}^n$, 
a positive integer $\ell$, a shift $x \in \mathbb{Z}^n$, and an $r \times n$ matrix $A \in \mathbb{Z}^{r \times n}$, 
satisfying the following two properties:
\begin{enumerate}
    \item[(1)] $\lambda^{(p)}(\mathcal{L}) \geq \ell^{1/p}$, and
    \item[(2)] Let 
    $$V := (x + \mathcal{L}) \cap B(0, \alpha \ell^{1/p})$$
    denote the set of all vectors in the coset $x + \mathcal{L}$ with $\ell_p$-norm at most $\alpha \ell^{1/p}$. 
    Let $A(V)$ denote the image of $V$ under the $\mathbb{Z}$-linear map $A: \mathbb{Z}^n \rightarrow \mathbb{Z}^r$. Then, 
    $$\{0, 1\}^r \subseteq A(V) := \{ A(v) : v \in V \}.$$
\end{enumerate}
\end{Def}

In applications, we typically choose $n$ to be polynomial in $r$ (more precisely, $r \sim n^{\epsilon}$ for some constant $\epsilon > 0$). 
Property (1) guarantees that the $\ell_p$-minimum distance of the lattice is sufficiently large. 
Property (2) ensures that not only $V$, but also its projection $A(V)$ under the linear map $A$, contains at least $2^r \sim 2^{n^{\epsilon}}$ elements. In particular, the cardinality of both sets is at least subexponential in $n$. 

In our de-randomization, we will choose $A$ to be the projection map of $V$ onto the first $r$ coordinates, 
that is, $A(z_1, \dots, z_n)^T = (z_1, \dots, z_r)^T$. Mathematically, this corresponds to the matrix block representation:
$$\mathbf{A} = [I_r \mid 0_{r \times (n-r)}]$$ 
where $I_r$ is the $r \times r$ identity matrix and $0_{r \times (n-r)}$ denotes the $r \times (n-r)$ zero matrix. 

If we only require Property (1) and the weaker condition that the set $V$ itself has subexponentially many elements (omitting the projection constraint on $A(V)$), then $\mathcal{L}$ is called a \emph{locally dense lattice in the weaker sense}. If we require both Property (1) and Property (2) as defined above, then $\mathcal{L}$ is a \emph{locally dense lattice in the strong sense}. 
Generally, a locally dense lattice in the weaker sense can be shown to be locally dense in the strong sense with some additional work. Thus, constructing a locally dense lattice in the weaker sense is the most critical initial milestone.

The following reduction is established in ~\cite[Theorem 5.1]{Mic12}, see also ~\cite[Corollary 2.11]{BP23}. 
 
\begin{Thm}\label{thm:reduction} 
Let $1\leq p < \infty$, let $r$ be a positive integer, and let $\alpha \in (2^{-1/p}, 1)$. Suppose that 
there exists an algorithm computing a $(p, \alpha, r, \mathrm{poly}(r))$-locally dense lattice in deterministic $\mathrm{poly}(r)$ time. 
Then, $\gamma$-$\mathrm{GapSVP}_p$ is $\mathbf{NP}$-hard for all $1 \leq \gamma < 1/\alpha$. 
\end{Thm}
 
 In this theorem, we can remove a finite number of small integers $r$ and assume that $r$ is sufficiently large. Proving the deterministic $\mathbf{NP}$-hardness of $\gamma$-$\mathrm{GapSVP}_p$ for $1 \leq \gamma < 1/\alpha$ 
 is thus reduced to a deterministic polynomial time algorithm 
 which computes a locally dense lattice as above for all sufficiently large $r$. The existence of such a deterministic polynomial time algorithm is conjectured 
 by Micciancio \cite[Conjecture 1]{Mic14}. It is the last part of his derandomization program.

In the existing literature, several techniques have been developed to construct locally dense lattices, but none are deterministic. 
The central paradigm in~\cite{BP23} is to construct the desired locally dense lattices via the lifting of Reed-Solomon codes. 
This is the framework upon which we build, and we recall its explicit construction in the next section.

\section{Reed-Solomon lattices}

Let $q$ be a prime, and $\mathbb{F}_q = \mathbb{Z}/q\mathbb{Z}$ be the prime finite field of $q$ elements. Write 
$\mathbb{F}_q = \{a_1, \dots, a_q\}$. This ordering of the elements in $\mathbb{F}_q$ will be fixed. 
For a positive integer $1 < k < q$, let 
$$H_q(k) := 
\begin{pmatrix}
1 & 1 & \dots & 1 \\
 a_1 & a_2 & \dots & a_q \\
 \vdots & \vdots & \ddots & \vdots \\
a_1^{k-1} & a_2^{k-1} & \dots & a_q^{k-1} 
\end{pmatrix}.$$
The transposes of the $k$ row vectors of $H_q(k)$ generate the $k$-dimensional Reed-Solomon code $\mathrm{RS}_q(k)$ in $\mathbb{F}_q^q$. We shall work with the dual code of $\mathrm{RS}_q(k)$, which is the $(q-k)$-dimensional Reed-Solomon code $\mathrm{RS}_q(q-k)$ in $\mathbb{F}_q^q$. Thus, the $k \times q$ matrix $H_q(k)$ is the parity-check matrix of $\mathrm{RS}_q(q-k)$ in $\mathbb{F}_q^q$. That is, 
$$\mathrm{RS}_q(q-k) = \{ \mathbf{x} \in \mathbb{F}_q^q : H_q(k)\mathbf{x} = \mathbf{0} \}.$$
It is a $(q-k)$-dimensional MDS code in $\mathbb{F}_q^q$, and hence has minimum distance $q - (q-k) + 1 = k+1$. 

\begin{Def} 
Given a prime $q$ and an integer $1 < k < q$, the Reed-Solomon lattice $\mathcal{L}_{q,k}$ in $\mathbb{Z}^q$ is the lifting to $\mathbb{Z}^q$ of the Reed-Solomon code $\mathrm{RS}_q(q-k)$ in $\mathbb{F}_q^q$: 
$$\mathcal{L}_{q,k} := \mathrm{RS}_q(q-k) + q\mathbb{Z}^q = \{ \mathbf{v} \in \mathbb{Z}^q : H_q(k)\mathbf{v} = \mathbf{0}\} \supseteq q\mathbb{Z}^q.$$
The dimension of the lattice $\mathcal{L}_{q,k}$ in $\mathbb{Z}^q$ is $n=q$. 
\end{Def}

A generator matrix for $\mathcal{L}_{q,k}$ can be found in deterministic poly($q$) time as follows.  
Partition the parity-check matrix as
\[
H_q(k)=\begin{pmatrix}V&W\end{pmatrix},
\]
where
\[
V=
\begin{pmatrix}
1&1&\cdots&1\\
a_1&a_2&\cdots&a_k\\
\vdots&\vdots&&\vdots\\
a_1^{k-1}&a_2^{k-1}&\cdots&a_k^{k-1}
\end{pmatrix}
\in \mathbb{F}_q^{k\times k}
\]
and
\[
W=
\begin{pmatrix}
1&\cdots&1\\
a_{k+1}&\cdots&a_q\\
\vdots&&\vdots\\
a_{k+1}^{k-1}&\cdots&a_q^{k-1}
\end{pmatrix}
\in \mathbb{F}_q^{k\times(q-k)}.
\]
Since $a_1,\dots,a_k$ are distinct,
\[
\det(V)=\prod_{1\le i<j\le k}(a_j-a_i)\ne 0
\quad\text{in }\mathbb{F}_q,
\]
so $V$ is invertible.  Set
\[
\Lambda:=V^{-1}W\in\mathbb{F}_q^{k\times(q-k)},
\]
and choose an integer lift
\[
\widehat{\Lambda}\in\mathbb{Z}^{k\times(q-k)}
\quad\text{with}\quad
\widehat{\Lambda}\bmod q=\Lambda.
\]
For example, every entry of $\widehat{\Lambda}$ may be chosen in
$\{0,1,\dots,q-1\}$, or as a centered residue representative.

\begin{Prop}
With basis vectors placed in the columns, the matrix
\[
\boxed{
B_{q,k}
=
\begin{pmatrix}
qI_k&-\widehat{\Lambda}\\[2mm]
0_{(q-k)\times k}&I_{q-k}
\end{pmatrix}
}
\]
is an integer basis matrix for $L_{q,k}$.  In particular,
\[
L_{q,k}=B_{q,k}\mathbb{Z}^q
\qquad\text{and}\qquad
\det(L_{q,k})=|\det(B_{q,k})|=q^k.
\]
\end{Prop}

\begin{proof}
Modulo $q$, one has
\begin{align*}
H_q(k)B_{q,k}
&=
\begin{pmatrix}V&W\end{pmatrix}
\begin{pmatrix}
qI_k&-\widehat{\Lambda}\\
0&I_{q-k}
\end{pmatrix}\\
&=
\begin{pmatrix}
qV&W-V\widehat{\Lambda}
\end{pmatrix}
\equiv
\begin{pmatrix}
0&W-V(V^{-1}W)
\end{pmatrix}
=0.
\end{align*}
Thus every column of $B_{q,k}$ belongs to $L_{q,k}$, and therefore
\[
B_{q,k}\mathbb{Z}^q\subseteq L_{q,k}.
\]

Conversely, let
\[
v=
\begin{pmatrix}u\\w\end{pmatrix}
\in L_{q,k},
\qquad
u\in\mathbb{Z}^k,
\quad
w\in\mathbb{Z}^{q-k}.
\]
The syndrome condition gives
\[
Vu+Ww\equiv 0\pmod q.
\]
Multiplying by $V^{-1}$ over $\mathbb{F}_q$ yields
\[
u+\Lambda w\equiv 0\pmod q.
\]
Since $\widehat{\Lambda}$ is an integer lift of $\Lambda$, there is a
vector $z\in\mathbb{Z}^k$ such that
\[
u=qz-\widehat{\Lambda}w.
\]
Hence
\[
v
=
\begin{pmatrix}
qz-\widehat{\Lambda}w\\
w
\end{pmatrix}
=
B_{q,k}
\begin{pmatrix}z\\w\end{pmatrix},
\]
which proves $L_{q,k}\subseteq B_{q,k}\mathbb{Z}^q$.

Finally, $B_{q,k}$ is block upper triangular, so
\[
\det(B_{q,k})
=
\det(qI_k)\det(I_{q-k})
=q^k.
\]
\end{proof}

By definition, a column vector $\mathbf{v} \in \mathbb{Z}^q$ is in $\mathcal{L}_{q,k}$ if and only if its reduction $\mathbf{v} \pmod{q}$ is in $\mathrm{RS}_q(q-k)$.
Since the minimum Hamming distance of $\mathrm{RS}_q(q-k)$ is $k+1$, it follows that the $\ell_p$-minimum distance of $\mathcal{L}_{q,k}$ satisfies 
$$\lambda^{(p)}(\mathcal{L}_{q,k}) \geq (k+1)^{1/p}.$$
This is true for ${\bf v} \in \mathcal{L}_{q,k}-\{ 0\}$ if ${\bf v}$ is not divisible by $q$. If ${\bf v}\not=0$ is divisible by $q$, we also have 
\[ |{\bf v}|_p \geq q^{1/p}  \geq (k+1)^{1/p}.\]
The exact minimum distance $\lambda^{(p)}(\mathcal{L}_{q,k})$ of the lattice $\mathcal{L}_{q,k}$ is unknown. This leads to the following natural question. 

\begin{Question} What is the minimum distance $\lambda^{(p)}(\mathcal{L}_{q,k})$?  
\end{Question}

In constructing locally dense lattices, 
we would like to have the minimum distance 
to be large. In this direction, 
an important observation by Bennett and Peikert~\cite{BP23} is the following improvement. 

\begin{Lem}\label{lem:BP-min-dist} 
Let $1 \leq k \leq q/2$. Then, for all $1 \leq p < \infty$, we have 
$$\lambda^{(p)}(\mathcal{L}_{q,k}) \geq (2k)^{1/p}.$$
\end{Lem}

\begin{proof} 
For the reader's convenience, we recall the elegant proof from~\cite{BP23}. Since $\|\mathbf{x}\|_p \geq \|\mathbf{x}\|_1^{1/p}$ for all $\mathbf{x} \in \mathbb{Z}^q$, the lemma is reduced to the case $p=1$. 
Suppose now that $\mathbf{x} = (x_1, \dots, x_q)^T \in \mathcal{L}_{q,k}$ is a lattice vector with 
$$m := \|\mathbf{x}\|_1 = \sum_{i=1}^q |x_i| < 2k.$$
To prove the lemma, we need to show that $\mathbf{x} = \mathbf{0}$. Consider the syndrome equation $H_q(k)\mathbf{x} = \mathbf{0}$, which is the following system in $\mathbb{F}_q$:
$$\sum_{i=1}^q x_i a_i^j = 0, \quad 0 \leq j \leq k-1,$$
where $0^0 = 1$ by convention. 
Let $T^+$ be the multiset consisting of all $a_i$ with multiplicity $x_i$ when $x_i > 0$. Let $T^-$ be the multiset consisting of all $a_i$ with multiplicity $-x_i$ when $x_i < 0$. Thus, $|T^+| + |T^-| = m < 2k \leq q$. 
The above system with $j=0$ says that $|T^+| - |T^-| \equiv 0 \pmod{q}$. This implies that $|T^+| = |T^-| = m/2 < k$, since both $|T^+|$ and $|T^-|$ are non-negative integers bounded by $m < 2k \leq q$. 
The system now implies that the $j$-th power symmetric functions of the two multisets $T^+$ and $T^-$ are identical for all $1 \leq j \leq k-1$. By Newton's identities, the $j$-th elementary symmetric functions of the two multisets $T^+$ and $T^-$ (of the same cardinality $m/2 \leq k-1$) are identical for all $1 \leq j \leq k-1$. It follows that the two multisets $T^+$ and $T^-$ must be identical, which forces them to be empty since they are disjoint. This means that $\mathbf{x} = \mathbf{0}$, completing the proof. 
\end{proof}

We would like to show that $\mathcal{L}_{q,k}$ is a $(p, \alpha, \lfloor q^{\epsilon}\rfloor, q)$-locally dense lattice in $\mathbb{Z}^q$ for each constant $\alpha \in (2^{-1/p}, 1)$ and for suitable choices of $1 \leq k < q/2$ and $0 < \epsilon < 1$. Most importantly, we need to deterministically construct and prove an explicit bad center.

\begin{Def}
Let $1 \leq h \leq q/2$ be an integer. Take a binary vector  $\mathbf{y}  \in \mathbb{Z}^q$ with Hamming weight $h$. For simplicity, we choose and fix 
 $$\mathbf{y} = (1, \dots, 1, 0, \dots, 0)^T \in \mathbb{Z}^q$$ 
 with the first $h$ coordinates being $1$ and the last $q-h$ coordinates being $0$. Write 
$$\mathbf{u} := H_q(k)\mathbf{y} = (h, h_1, \dots, h_{k-1})^T \in \mathbb{F}_q^k.$$ 
\end{Def}
Clearly, the first coordinate of $\mathbf{u}$ is $h$, as the first row of $H_q(k)$ is $(1, 1, \dots, 1)$. By definition, we obtain the following property. 

\begin{Prop} 
The coset of the binary vector $\mathbf{y}$ for the lattice $\mathcal{L}_{q,k}$ is given by 
$$\mathbf{y} + \mathcal{L}_{q,k} = \{ \mathbf{x} = (x_1, \dots, x_q)^T \in \mathbb{Z}^q : H_q(k)\mathbf{x} = \mathbf{u} = H_q(k)\mathbf{y}\}.$$
\end{Prop}

We shall prove that this explicit vector $\mathbf{y}$ is a bad center of the locally dense Reed-Solomon lattice $\mathcal{L}_{q,k}$ for suitable choices of $k$ and $h = \lfloor (1+\epsilon)k \rfloor \leq q/2$. To show that the coset $\mathbf{y} + \mathcal{L}_{q,k}$ contains many short vectors, it is sufficient to show that it contains many short binary vectors. For this purpose, we introduce the following subset $S_2(\mathbf{y}, h)$, where the subscript $2$ indicates that we restrict our attention to binary vectors. 

\begin{Def} 
Let $S_2(\mathbf{y}, h)$ denote the set of binary vectors in the coset $\mathbf{y} + \mathcal{L}_{q,k}$ with Hamming weight $h$: 
$$S_2(\mathbf{y}, h) := \{ \mathbf{x} = (x_1, \dots, x_q)^T \in \{0, 1\}^q : H_q(k)\mathbf{x} = \mathbf{u}, \ \|\mathbf{x}\|_1 = h\}.$$ 
Clearly, $\mathbf{y} \in S_2(\mathbf{y}, h)$. Fix a constant $\alpha \in (2^{-1/p}, 1)$, so that $1 < 2\alpha^p < 2$. Define 
$$h := \lfloor \alpha^p (2k) \rfloor = \lfloor (2\alpha^p) k \rfloor = \lfloor (1+\epsilon)k \rfloor, \quad 0 < \epsilon := 2\alpha^p - 1 < 1.$$ 
In particular, $\alpha = \left(\frac{1+\epsilon}{2}\right)^{1/p}$. 
\end{Def}

Every vector $\mathbf{x} \in S_2(\mathbf{y}, h)$ is a binary vector of Hamming weight $h$, and thus its $\ell_p$-norm is bounded by 
$$\|\mathbf{x}\|_p = h^{1/p} \leq \alpha (2k)^{1/p} \leq \alpha \lambda^{(p)}(\mathcal{L}_{q,k}),$$
where the last inequality follows from Lemma \ref {lem:BP-min-dist}. This yields the following proposition. 

\begin{Prop}\label{prop:V-contains-S2} 
In the notation of Definition~\ref{def:locally-dense} on locally dense lattices, we take $\mathcal{L} = \mathcal{L}_{q,k}$. Then $\ell = 2k$ satisfies Condition (1) of Definition~\ref{def:locally-dense}, and 
$$V := (\mathbf{y} + \mathcal{L}_{q,k}) \cap B(0, \alpha (2k)^{1/p}) \supseteq S_2(\mathbf{y}, h).$$
\end{Prop}

In order to show that $V$ satisfies Condition (2) of Definition~\ref{def:locally-dense}, it is sufficient to show that the linear projection to the first $r$ coordinates, when restricted to the subset $S_2(\mathbf{y}, h)$, is surjective onto $\{0, 1\}^r$. That is, we want to show that 
$$\{0, 1\}^r \subseteq A(S_2(\mathbf{y}, h)).$$ 
To achieve this, we first treat the simpler case of estimating $|S_2(\mathbf{y}, h)|$, showing that it is at least subexponential in $n=q$. The stronger version incorporating the linear projection will be handled in the final section. 

\begin{Rmk} 
If one applies the Weil bound to estimate $|S_2(\mathbf{y}, h)|$, one only obtains non-trivial information for $\epsilon > 1$, which implies $\alpha > 1$. However, this is of no help for our goals. By Theorem~\ref{thm:reduction}, we must work with a constant $\alpha \in (2^{-1/p}, 1)$ (i.e., $0 < \epsilon < 1$) to prove the conjecture of van Emde Boas. Furthermore, to de-randomize the main result in~\cite{BP23}, we need to take the constant $\alpha$ arbitrarily close to $2^{-1/p}$, which means taking $\epsilon$ to be an arbitrarily small positive constant. 
\end{Rmk}

As mentioned, our first goal is to show that the cardinality of the set $S_2(\mathbf{y}, h)$ is at least subexponential in $q$; that is, $|S_2(\mathbf{y}, h)| \geq q^{\delta k}$ for some constant $\delta > 0$, where $h = \lfloor(1+\epsilon)k\rfloor$ and $k = \lfloor q^{\epsilon_1}\rfloor$ with $0 < \epsilon_1 < 1$. To do so, we relate the quantity $|S_2(\mathbf{y}, h)|$ to the number of $\mathbb{F}_q$-rational points on a certain quasi-projective algebraic variety over $\mathbb{F}_q$, and then apply deep results from arithmetic geometry, specifically Deligne's theorem on the Weil conjectures. 

Recall that $\mathbb{F}_q = \{a_1, \dots, a_q\}$ with a fixed ordering. Let $\mathbf{x} = (x_1, \dots, x_q)^T \in S_2(\mathbf{y}, h)$. Since $\mathbf{x}$ is a binary vector with exactly $h$ coordinates equal to $1$ and the rest equal to $0$, we can write $x_{i_j} = 1$ for $j = 1, \dots, h$, where $1 \leq i_1 < \dots < i_h \leq q$. These indices $\{i_1, \dots, i_h\}$ uniquely determine the elements $\{a_{i_1}, \dots, a_{i_h}\}$ in $\mathbb{F}_q$. By definition of the matrix $H_q(k)$, we compute 
$$H_q(k)\mathbf{x} = \left(h, \sum_{j=1}^h a_{i_j}, \sum_{j=1}^h a_{i_j}^2, \dots, \sum_{j=1}^h a_{i_j}^{k-1}\right)^T.$$
The first coordinate on both sides is equal to $h$. Thus, the syndrome equation 
$$H_q(k)\mathbf{x} = \mathbf{u} = \begin{pmatrix} h \\ h_1 \\ \vdots \\ h_{k-1} \end{pmatrix}$$ 
reduces to the following system of equations: 
$$\sum_{j=1}^h a_{i_j} = h_1, \quad \sum_{j=1}^h a_{i_j}^2 = h_2, \quad \dots, \quad \sum_{j=1}^h a_{i_j}^{k-1} = h_{k-1}.$$

\begin{Def} 
Let $W = \{ (1, a, \dots, a^{k-1})^T : a \in \mathbb{F}_q \} \subset \mathbb{F}_q^k$ be the set of the $q$ column vectors of the matrix $H_q(k)$. 
\end{Def}

\begin{Def} 
Let 
$$\mathcal{N}_h(q) := \left\{ (\mathbf{w}_1, \dots, \mathbf{w}_h) \in W^h : \sum_{i=1}^h \mathbf{w}_i = \mathbf{u} \right\}, \quad N_h(q) := |\mathcal{N}_h(q)|.$$ 
$$\mathcal{N}_h^*(q) := \left\{ (\mathbf{w}_1, \dots, \mathbf{w}_h) \in W^h : \sum_{i=1}^h \mathbf{w}_i = \mathbf{u}, \ \mathbf{w}_i \text{ distinct} \right\}, \quad N_h^*(q) := |\mathcal{N}_h^*(q)|.$$ 
\end{Def}

Each vector $\mathbf{x} \in S_2(\mathbf{y}, h)$ yields a unique solution $(\mathbf{w}_{i_1}, \dots, \mathbf{w}_{i_h}) \in W^h$ in $\mathcal{N}_h^*(q)$, where 
$$\mathbf{w}_{i_j} = \begin{pmatrix} 1 \\ a_{i_j} \\ \vdots \\ a_{i_j}^{k-1} \end{pmatrix} \in W, \quad 1 \leq j \leq h.$$ 
Since $i_1 < \dots < i_h$, this solution under the permutation of the indices yields $h!$ distinct solutions $(\mathbf{w}_1, \dots, \mathbf{w}_h)$ in $\mathcal{N}_h^*(q)$. Distinct vectors $\mathbf{x} \in S_2(\mathbf{y}, h)$ yield distinct solutions in $\mathcal{N}_h^*(q)$ because each $\mathbf{x}$ corresponds to a distinct subset $\{a_{i_1}, \dots, a_{i_h}\}$ of $\mathbb{F}_q$. Thus, each vector $\mathbf{x} \in S_2(\mathbf{y}, h)$ contributes exactly $h!$ to $N_h^*(q)$, giving the identity 
$$|S_2(\mathbf{y}, h)| = \frac{1}{h!} N_h^*(q).$$
To establish our main theorem, we must prove a sharp lower bound for $N_h^*(q)$. Clearly, $N_h^*(q) \leq N_h(q)$. A natural stepping stone is to first prove a large lower bound for $N_h(q)$. 

\begin{Def} 
Let $X_{k,h,\mathbf{u}}$ be the affine algebraic variety in the affine space $\mathbb{A}^h$ defined by 
$$X_{k,h,\mathbf{u}} : \sum_{i=1}^h x_i^j = h_j, \quad 1 \leq j \leq k-1.$$ 
Let $X_{k,h,\mathbf{u}}^*$ be the open subvariety of $X_{k,h,\mathbf{u}}$ consisting of points with pairwise distinct coordinates; that is, 
$$X_{k,h,\mathbf{u}}^* := X_{k,h,\mathbf{u}} \cap \bigcap_{1 \leq i_1 < i_2 \leq h} \{ x_{i_1} - x_{i_2} \neq 0 \}.$$
\end{Def}

This affine variety $X_{k,h,\mathbf{u}}$ is defined by $k-1$ equations in $h$ variables. Because we choose $h = \lfloor (1+\epsilon)k \rfloor$ to be significantly larger than $k$, one expects that the system has many $\mathbb{F}_q$-rational points for large $q$; roughly on the order of 
$$q^{h-(k-1)} = q^{\lfloor \epsilon k +1 \rfloor}$$ 
rational points, which is subexponential in $q$ for $k = \lfloor q^{\epsilon_1} \rfloor$. 

By definition, we observe that $N_h(q)$ is precisely the number of $\mathbb{F}_q$-rational points on $X_{k,h,\mathbf{u}}$. Similarly, $N_h^*(q)$ is the number of $\mathbb{F}_q$-rational points on $X_{k,h,\mathbf{u}}^*$. In summary, we have established the following relations: 
\begin{Prop} 
We have 
$$N_h(q) = |X_{k,h,\mathbf{u}}(\mathbb{F}_q)|, \quad N_h^*(q) = |X_{k,h,\mathbf{u}}^*(\mathbb{F}_q)|, \quad |S_2(\mathbf{y}, h)| = \frac{1}{h!} |X_{k,h,\mathbf{u}}^*(\mathbb{F}_q)|.$$ 
\end{Prop}

These cardinalities have been analyzed in the literature using Weil's bound for character sums; see, for example,~\cite{LMRW20} and~\cite{BP23}. However, methods relying solely on the Weil bound only yield non-trivial estimates when $\epsilon > 1$, making them too weak to handle our case of interest, $0 < \epsilon < 1$. 

To overcome this limitation, we will analyze these points using Deligne's theorem on the Weil conjectures. To apply this machinery, we must first understand the geometry of the affine variety $X_{k,h,\mathbf{u}}$ and its compactification, which turns out to be remarkably well-behaved. These varieties are complete intersections in both affine and projective spaces with very mild singularities, allowing the full strength of Deligne's theorem to be brought to bear. The more complex open subvariety $X_{k,h,\mathbf{u}}^*$ will then be handled in the subsequent sections via inclusion-exclusion sieving. 

% Add this to your preamble if not already present:
% \DeclareMathOperator{\Sing}{Sing}
% \DeclareMathOperator{\dimen}{dim} % or use standard \dim

\section{Geometry of the Variety $X_{k, h, \mathbf{u}}$}

\begin{Def} 
Let $\overline{X}_{k,h, \mathbf{u}}$ be the projective variety in the projective space $\mathbb{P}^h$ defined over $\mathbb{F}_q$ by 
\[
\overline{X}_{k,h, \mathbf{u}}: \sum_{i=1}^h x_i^j = h_j x_{h+1}^j, \quad 1 \leq j \leq k-1.
\]
Let $\overline{X}_{k,h}$ be the projective variety in $\mathbb{P}^{h-1}$ defined over $\mathbb{F}_q$ by 
\[
\overline{X}_{k,h}: \sum_{i=1}^h x_i^j = 0, \quad 1 \leq j \leq k-1.
\]
\end{Def}

The variety $\overline{X}_{k,h, \mathbf{u}}$ is the projective closure of the affine variety $X_{k,h, \mathbf{u}}$ in $\mathbb{P}^h$. Let $H_{\infty} \subset \mathbb{P}^h$ be the hyperplane at infinity defined by the equation $x_{h+1} = 0$. Under the natural identification of $H_{\infty}$ with $\mathbb{P}^{h-1}$, the projective variety $\overline{X}_{k,h}$ is realized as the intersection of $\overline{X}_{k,h,\mathbf{u}}$ with $H_{\infty}$:
\[
\overline{X}_{k,h} = \overline{X}_{k,h, \mathbf{u}} \cap H_{\infty}.
\]
Thus, we have the open-closed decomposition
\[
X_{k,h, \mathbf{u}} = \overline{X}_{k,h, \mathbf{u}} \setminus H_{\infty} = \overline{X}_{k,h, \mathbf{u}} \setminus \overline{X}_{k,h}.
\]
This reduces the point counting on $X_{k,h, \mathbf{u}}$ to estimating the number of $\mathbb{F}_q$-rational points on the two projective varieties $\overline{X}_{k,h, \mathbf{u}}$ and $\overline{X}_{k,h}$: 
\[
N_h(q) = |X_{k,h, \mathbf{u}}(\mathbb{F}_q)| = |\overline{X}_{k,h, \mathbf{u}}(\mathbb{F}_q)| - |\overline{X}_{k,h}(\mathbb{F}_q)|.
\]
To understand these two terms, we investigate the geometry of the projective varieties $\overline{X}_{k,h, \mathbf{u}}$ and $\overline{X}_{k,h}$.

\begin{Prop}\label{comp} 
Let $1 \leq k \leq h < q$ with $q$ being a prime. 
The projective variety $\overline{X}_{k,h, \mathbf{u}}$ over $\mathbb{F}_q$ is a complete intersection of dimension $h-k+1$ in $\mathbb{P}^h$. 
The projective variety $\overline{X}_{k,h}$ over $\mathbb{F}_q$ is a complete intersection of dimension $h-k$ in $\mathbb{P}^{h-1}$.  
\end{Prop}

\begin{proof}
Each time we intersect a projective variety with a hypersurface defined by a single homogeneous equation, the dimension drops by at most $1$. Since $\overline{X}_{k,h, \mathbf{u}}$ is cut out by $k-1$ homogeneous equations in $\mathbb{P}^h$, we have the lower bound
\[
\dim(\overline{X}_{k,h, \mathbf{u}}) \geq h - (k-1) = h-k+1.
\]
By definition, equality holds if and only if $\overline{X}_{k,h, \mathbf{u}}$ is a complete intersection. Thus, it suffices to prove the corresponding upper bound
\[
\dim(\overline{X}_{k,h, \mathbf{u}}) \leq h-k+1.
\]   
The dimension-dropping property implies the chain of inequalities
\[
\dim(\overline{X}_{h+1,h, \mathbf{u}}) \geq \dim(\overline{X}_{h,h, \mathbf{u}}) - 1 \geq \dots \geq \dim(\overline{X}_{k,h, \mathbf{u}}) - (h-k+1).
\]   
Hence, it is enough to show that the left-hand term is zero; that is, the variety $\overline{X}_{h+1,h, \mathbf{u}} \subset \mathbb{P}^h$ defined by the system
\[
\sum_{i=1}^h x_i^j = h_j x_{h+1}^j, \quad 1 \leq j \leq h
\]
is a zero-dimensional projective variety. 

We analyze the solutions in $\mathbb{P}^h(\overline{\mathbb{F}}_q)$ by splitting them into two cases based on the coordinate $x_{h+1}$:

\paragraph{Case 1 ($x_{h+1} = 0$):} The system reduces to $\sum_{i=1}^h x_i^j = 0$ for all $1 \leq j \leq h$. Because $h < q$, the integers $1, \dots, h$ are coprime to the characteristic of the field. By Newton's identities, the vanishing of these power sums implies that all elementary symmetric polynomials of the variables $x_1, \dots, x_h$ must vanish. This forces $x_1 = x_2 = \dots = x_h = 0$. Since the trivial point $(0, \dots, 0)$ does not define a point in the projective space $\mathbb{P}^h$, there are no solutions in this case.

\paragraph{Case 2 ($x_{h+1} \neq 0$):} We normalize the coordinate by setting $x_{h+1} = 1$. The system becomes:
\[
\sum_{i=1}^h x_i^j = h_j \quad (1 \leq j \leq h).
\]
Again appealing to Newton's identities (which is valid since $h < q$), the power sums $h_1, \dots, h_h$ uniquely determine the elementary symmetric polynomials $e_1, \dots, e_h$ of the variables $x_1, \dots, x_h$. Consequently, the coordinates $\{x_1, \dots, x_h\}$ are uniquely determined as the roots of the single polynomial:
\[
f(T) = T^h - e_1 T^{h-1} + e_2 T^{h-2} - \dots + (-1)^h e_h \in \mathbb{F}_q[T].
\]
Over the algebraic closure $\overline{\mathbb{F}}_q$, $f(T)$ has exactly $h$ roots (counted with multiplicity). Therefore, any coordinate vector $(x_1, \dots, x_h)$ must be a permutation of these roots, which yields at most $h!$ distinct points in the projective space.

Since $\overline{X}_{h+1,h,\mathbf{u}}$ contains at most $h!$ points over $\overline{\mathbb{F}}_q$, its dimension is $0$. 
This proves that $\overline{X}_{k,h,\mathbf{u}}$ is a complete intersection of dimension $h-k+1$. 
In particular, $\overline{X}_{h+1,h,\mathbf{u}}$ is a complete intersection of dimension $0$, defined by $h$ homogeneous polynomial equations of degrees $\{1, 2, \cdots, h\}$. Bezout's well known theorem then implies that $\overline{X}_{h+1,h,\mathbf{u}}$ has at most $h!$ points over $\overline{\mathbb{F}}_q$, 
in fact, exactly $h!$ points counting multiplicity, consistent with the above self-contained proof. 

The proof for the variety $\overline{X}_{k,h}$ follows by a similar induction-downward argument starting from $\mathbb{P}^{h-1}$.
\end{proof}

For a variety $X$, let $\Sing(X)$ denote the singular locus of $X$. If $X$ is smooth, then $\Sing(X) = \emptyset$ and we set $\dim \Sing(X) = -1$ by convention. 

\begin{Lem}\label{singular} 
Let $X$ be a complete intersection of dimension $m \geq 0$ in some projective space $\mathbb{P}^n$. 
Let $Z$ be a hyperplane in $\mathbb{P}^n$. Assume that the hyperplane section $X \cap Z$ is a complete intersection of dimension $m-1$ in $\mathbb{P}^n$. Then, 
\[
\lvert \dim \Sing(X) - \dim \Sing(X \cap Z) \rvert \leq 1.
\]
\end{Lem}  

\begin{proof} 
The upper bound on the hyperplane section,
\[
\dim \Sing(X \cap Z) \leq \dim \Sing(X) + 1,
\]
is Zak's lemma; a proof is provided in Katz's appendix to \cite{Hoo91}. The converse inequality,
\[
\dim \Sing(X) \leq \dim \Sing(X \cap Z) + 1,
\]
is established in \cite[Lemma 3]{Kat99}. 
\end{proof}
  
\begin{Prop}\label{smooth} 
Let $1 \leq k \leq h < q$ with $q$ being a prime. 
The projective variety $\overline{X}_{k,h}$ over $\mathbb{F}_q$ is smooth. In particular, the projective variety $\overline{X}_{k,h, \mathbf{u}}$ over $\mathbb{F}_q$ has at most finitely many isolated singularities. 
\end{Prop}  

\begin{proof} 
For $k=1$, $\overline{X}_{1,h} = \mathbb{P}^{h-1}$ is trivially smooth. Assume that $2 \leq k \leq h$. 
Let $(x_1, \dots, x_h)$ be a singular point on the projective variety $\overline{X}_{k,h}$. 
By the Jacobian criterion, the $(k-1) \times h$ Jacobian matrix 
\[
J = \begin{pmatrix} 
1 & 1 & \dots & 1 \\
2x_1 & 2x_2 & \dots & 2x_h \\
\vdots & \vdots & \ddots & \vdots \\
(k-1)x_1^{k-2} & (k-1)x_2^{k-2} & \dots & (k-1)x_h^{k-2} 
\end{pmatrix}
\]
must have rank strictly less than $k-1$. Since $k \leq h < q$, the coefficients $j$ for $1 \leq j \leq k-1$ are non-zero in $\mathbb{F}_q$. 

If $\operatorname{rank}(J) < k-1$, then any subset of $k-1$ columns of $J$ must be linearly dependent. Let $\{i_1, \dots, i_{k-1}\} \subseteq \{1, \dots, h\}$ be any choice of $k-1$ column indices. The determinant of the submatrix formed by these columns is:
\[
\det \begin{pmatrix} 
1 & \dots & 1 \\
2x_{i_1} & \dots & 2x_{i_{k-1}} \\
\vdots & \ddots & \vdots \\
(k-1)x_{i_1}^{k-2} & \dots & (k-1)x_{i_{k-1}}^{k-2} 
\end{pmatrix}
= (k-1)! \prod_{1 \leq a < b \leq k-1} (x_{i_b} - x_{i_a}).
\]
Because $k-1 < q$, the factorial $(k-1)!$ is invertible in $\mathbb{F}_q$. Consequently, this determinant vanishes if and only if $x_{i_a} = x_{i_b}$ for some $a < b$. Since this must hold for every choice of $k-1$ coordinates, the set $\{x_1, \dots, x_h\}$ can contain at most $k-2$ distinct values.

Since $(x_1, \dots, x_h)$ defines a point in the projective space, at least one of these values must be non-zero. Let $\{ z_1, \dots, z_e \}$ be the distinct non-zero elements in $\{x_1, \dots, x_h\}$ with corresponding positive multiplicities $\{m_1, \dots, m_e\}$. 
Thus, we have $1 \leq m_1 + \dots + m_e \leq h < q$ and $e \leq k-2$. Because $(x_1, \dots, x_h)$ is on the variety $\overline{X}_{k,h}$, its non-zero coordinates must satisfy:
\[
\sum_{i=1}^e m_i z_i^j = 0, \quad 1 \leq j \leq k-1.
\]  
Since $e \leq k-2 \leq k-1$, we can consider the first $e$ equations:
\[
\begin{pmatrix} 
z_1 & z_2 & \dots & z_e \\
z_1^2 & z_2^2 & \dots & z_e^2 \\
\vdots & \vdots & \ddots & \vdots \\
z_1^e & z_2^e & \dots & z_e^e
\end{pmatrix}
\begin{pmatrix} 
m_1 \\ m_2 \\ \vdots \\ m_e
\end{pmatrix}
= \mathbf{0}.
\]
Factoring out $z_i \neq 0$ from each column yields a standard Vandermonde matrix:
\[
\begin{pmatrix} 
1 & 1 & \dots & 1 \\
z_1 & z_2 & \dots & z_e \\
\vdots & \vdots & \ddots & \vdots \\
z_1^{e-1} & z_2^{e-1} & \dots & z_e^{e-1}
\end{pmatrix}
\begin{pmatrix} 
m_1 z_1 \\ m_2 z_2 \\ \vdots \\ m_e z_e
\end{pmatrix}
= \mathbf{0}.
\]
Since the $z_i$'s are distinct, this Vandermonde matrix is non-singular, which forces $m_i z_i = 0$ in $\mathbb{F}_q$ for all $1 \leq i \leq e$. Since $z_i \neq 0$, we must have $m_i = 0$ in $\mathbb{F}_q$. However, this directly contradicts our assumption that $1 \leq m_i \leq h < q$ for each $i$. Thus, no such singular point can exist, proving that $\overline{X}_{k,h}$ is smooth.

Now, $\overline{X}_{k,h}$ is the hyperplane section $\{ x_{h+1} = 0 \}$ of $\overline{X}_{k,h, \mathbf{u}}$. Both are complete intersections of dimension $h-k$ and $h-k+1$, respectively, in $\mathbb{P}^h$. By Lemma \ref{singular}, $\dim \Sing(\overline{X}_{k,h, \mathbf{u}})$ differs from $\dim \Sing(\overline{X}_{k,h})$ by at most $1$. Since $\overline{X}_{k,h}$ is smooth, its singular locus has dimension $-1$, which implies $\dim \Sing(\overline{X}_{k,h, \mathbf{u}}) \leq 0$. Therefore, the projective variety $\overline{X}_{k,h, \mathbf{u}}$ contains at most finitely many isolated singular points.
\end{proof} 

 \section{Rational Points on the Variety $X_{k,h, \mathbf{u}}$}

We need the following fundamental result on the number  $|X(\mathbb{F}_q)|$ of $\mathbb{F}_q$-rational points on a singular complete intersection $X$ defined over a finite field.  

\begin{Prop}\label{Del} 
Let $X$ be a projective complete intersection of dimension $m\geq 1$ in some projective space $\mathbb{P}^n$ defined over the finite field $\mathbb{F}_q$. Let $s = \dim \operatorname{Sing}(X)$. Then, for every integer $e\geq 1$, we have 
\[
\left| |X(\mathbb{F}_{q^e})| - \frac{q^{e(m+1)}-1}{q^e-1}\right| \leq C q^{\frac{e(m+1+s)}{2}},
\]
where $C$ is a constant independent of the extension degree $e$.   
\end{Prop}  
  
In the smooth case $s=-1$, this result is a direct consequence of Deligne's theorem \cite{De81} on the Weil conjectures.  
The general singular case can also be derived from Deligne's theorem and properties of $\ell$-adic cohomology, as noted by Hooley and Katz \cite{Hoo91}. If $s \geq m-1$, the estimate is trivial. If $s\leq m-2$ (the interesting case), $X$ is non-singular in codimension $1$, and hence a normal variety.  Since $X$ is a positive dimensional complete intersection, $X$ is geometrically connected. It follows that the variety $X$ is geometrically irreducible. The estimate is strongest when $s=-1$, giving the optimal square root estimate for 
the error term. The estimate becomes increasingly weaker when $s$ grows, i.e., when $X$ becomes more singular. 
  
The above constant $C$ can be taken to be the total degree of the zeta function of $X$, or the larger total $\ell$-adic Betti number of $X$. They can be explicitly bounded in terms of $n$ and the multi-degrees of the defining equations for $X$. We do not give the projective bound here. Instead, we will provide and use the simpler affine bound later. 
  
Now, we are ready to estimate the number $N_h(q)$ of $\mathbb{F}_q$-rational points on the affine complete intersection $X_{k,h, \mathbf{u}}$. 
 
\begin{Prop}\label{estimate1} 
Let $2\leq k < h < q$ with $q$ being a prime. Then, 
\[
|N_h(q) - q^{h-k+1}| \leq \frac{1}{2} (2k)^h q^{\frac{h-k+2}{2}}.
\]
\end{Prop} 
 
\begin{proof}
We give a self-contained proof using Proposition \ref{Del} as a black box and the total Betti number estimate in \cite[Theorem 1.2.1]{WZ26}.  
 
The variety $\overline{X}_{k,h}$ is a smooth projective complete intersection of dimension $h-k>0$ in $\mathbb{P}^{h-1}$. Proposition \ref{Del} with $s=-1$ yields the following estimate for all $e\geq 1$: 
\[
\left| |\overline{X}_{k,h}(\mathbb{F}_{q^e})| - \frac{q^{e(h-k+1)}-1}{q^e-1}\right| \leq C_1q^{\frac{e(h-k)}{2}},
\]
where the constant $C_1$ is independent of the extension degree $e$.   
  
The variety $\overline{X}_{k,h, \mathbf{u}}$ is a projective complete intersection of dimension $h-k+1 \geq 2$ in $\mathbb{P}^h$ with at most isolated singularities. Proposition \ref{Del} with $s=0$ yields the following estimate  for all $e\geq 1$: 
\[
\left| |\overline{X}_{k,h, \mathbf{u}}(\mathbb{F}_{q^e})| - \frac{q^{e(h-k+2)}-1}{q^e-1}\right| \leq C_2q^{\frac{e(h-k+2)}{2}},
\]
where the constant $C_2$ is independent of the extension degree $e$ .

Since $|X_{k,h, \mathbf{u}}(\mathbb{F}_{q^e})| = |\overline{X}_{k,h, \mathbf{u}}(\mathbb{F}_{q^e})| - |\overline{X}_{k,h}(\mathbb{F}_{q^e}) |$, we conclude that for any positive integer $e$, 
\begin{equation}\label{bound}
\left| |X_{k,h, \mathbf{u}}(\mathbb{F}_{q^e})|- q^{e(h-k+1)}\right| \leq (C_1+C_2)q^{\frac{e(h-k+2)}{2}}.
\end{equation}
Write the rational zeta function of $X_{k,h, \mathbf{u}}$ over $\mathbb{F}_q$ in reduced form:
\[
Z(X_{k,h, \mathbf{u}}, T) := \exp\left(\sum_{e=1}^{\infty} \frac{|X_{k,h, \mathbf{u}}(\mathbb{F}_{q^e})|}{e}T^e\right) = \frac{1}{(1-q^{h-k+1}T)}\frac{\prod_{i=1}^{d_1}(1-\alpha_iT)}{\prod_{j=1}^{d_2}(1-\beta_jT)}.
\]
Equivalently, for every positive integer $e$, we have 
\[
|X_{k,h, \mathbf{u}}(\mathbb{F}_{q^e})|- q^{e(h-k+1)} = \sum_{j=1}^{d_2} \beta_j^e - \sum_{i=1}^{d_1} \alpha_i^e.
\] 
The bound in \eqref{bound} for all $e$ implies that the following rational function 
\[
\sum_{e=0}^{\infty} \left(\sum_{j=1}^{d_2} \beta_j^e - \sum_{i=1}^{d_1} \alpha_i^e\right)T^e = \sum_{j=1}^{d_2}\frac{1}{1-\beta_jT} - \sum_{i=1}^{d_1}\frac{1}{1-\alpha_iT}
\] 
is analytic (possesses no poles) in the open disk $|T| < q^{-\frac{h-k+2}{2}}$. It follows that 
\[
|\alpha_i| \leq q^{\frac{h-k+2}{2}}, \quad |\beta_j| \leq q^{\frac{h-k+2}{2}}, \quad 1\leq i\leq d_1, \quad 1\leq j \leq d_2.
\]
This implies that for all positive integers $e$, 
\[
\left| |X_{k,h, \mathbf{u}}(\mathbb{F}_{q^e})|- q^{e(h-k+1)} \right| \leq \sum_{j=1}^{d_2} |\beta_j|^e + \sum_{i=1}^{d_1} |\alpha_i|^e \leq (d_1+d_2)q^{\frac{e(h-k+2)}{2}},
\] 
where $d_1+d_2$ is bounded by the total degree $d_1+d_2+1$ of $Z(X_{k,h, \mathbf{u}}, T)$. 

Now, the affine variety $X_{k,h,\mathbf{u}}$ in $\mathbb{A}^h$ is defined by the vanishing of $k-1$ equations with degrees at most $k-1$. Bombieri's total degree bound \cite{Bo78} for an affine variety gives 
\[
d_1 + d_2 +1 \leq (4(k-1)+9)^{h+k-1} = (4k+5)^{h+k-1}.
\]

This total degree bound can be improved via cohomological methods. For a prime $\ell \neq q$, let $\mathbb{Q}_{\ell}$ denote the field of $\ell$-adic rational numbers. The $\ell$-adic cohomological trace formula states that  
\[
Z(X_{k, h, \mathbf{u}}, T) = \prod_{i=0}^{2(h-k+1)} \det \left(I - \sigma_q T \mid H_c^i( X_{k, h, \mathbf{u}} \otimes \bar{\mathbb{F}}_q, \mathbb{Q}_{\ell})\right)^{(-1)^{i-1}},
\]
where $H_c^i( X_{k, h, \mathbf{u}} \otimes \bar{\mathbb{F}}_q, \mathbb{Q}_{\ell})$ denotes the $\ell$-adic cohomology with compact support, which is a finite-dimensional vector space over $\mathbb{Q}_{\ell}$, and $\sigma_q$ is the geometric Frobenius which induces an invertible linear map on this space. From the trace formula, one deduces the inequality 
\[
d_1 + d_2 + 1 \leq B_c(X_{k, h, \mathbf{u}}) := \sum_{i=0}^{2(h-k+1)} \dim_{\mathbb{Q}_{\ell}} H_c^i( X_{k, h, \mathbf{u}} \otimes \bar{\mathbb{F}}_q, \mathbb{Q}_{\ell}).
\]
The quantity $B_c(X_{k, h, \mathbf{u}})$ is the total Betti number, whose independence on the choice of the auxiliary prime $\ell$ remains a major open problem in arithmetic geometry. 

In any case, a sharp upper bound for the total Betti number yields a sharp upper bound for the total degree. For instance, Katz's estimate \cite[p. 43]{Kat01} for the total $\ell$-adic Betti number gives 
\[
d_1+d_2+1 \leq B_c(X_{k, h, \mathbf{u}}) \leq 3(2+k-1)^{h+k-1} = 3(k+1)^{h+k-1},
\]
which is better than Bombieri's bound. Further 
improved total Betti number bounds are studied systematically in \cite{WZ26}. Since $X_{k,h,\mathbf{u}}$ is an affine complete intersection in $\mathbb{A}^h$, we can apply Proposition \ref{betti} below (which is \cite[Theorem 1.2.2]{WZ26}) to get an asymptotically optimal total $\ell$-adic Betti number bound. This yields:
\[
d_1+d_2+1 \leq B_c(X_{k, h, \mathbf{u}}) \leq \binom{h-1}{k-2} k^h < 2^{h-1}k^h = \frac{1}{2}(2k)^h.
\] 
The proof is complete.
\end{proof}
 
\begin{Prop}[\cite{WZ26}] \label{betti}
Let $X$ be an affine complete intersection of codimension $r$ in $\mathbb{A}^n$ defined by $r\leq n$ polynomials of degrees at most $d$. Then, the total $\ell$-adic Betti number $B_c(X)$ of $X$ (and consequently, the total degree of the zeta function of $X$ if $X$ is defined over $\mathbb{F}_q$) is bounded by 
\[
B_c(X) \leq \binom{n-1}{r-1}(d+1)^n.
\] 
\end{Prop}
  
 As a consequence of Proposition \ref{estimate1}, we give a set of conditions which guarantees that $N_h(q)$ is subexponential in $q$ for all sufficiently large 
 $q$, that is, for all $q$ larger than an effective constant. Throughout the paper, all sufficiently large bounds are effective. 
 These conditions are what we need in constructing locally dense lattices. 
 \begin{Thm}\label{weakloc} 
Let $\epsilon>0$. Choose a positive constant $\epsilon_1$ and positive integers $k, h$ such that 
\[ 0 < \epsilon_1 < \frac{\epsilon}{2(1+\epsilon)} < \frac{1}{2}, \quad 2k = 2\lfloor \frac{q^{\epsilon_1}}{2}\rfloor, \quad h = \lfloor (1+\epsilon)k\rfloor. \]
Then, for all sufficiently large prime $q$, we have 
\[ N_h(q) \geq \frac{1}{2} q^{\epsilon k}. \]
\end{Thm}
 
 \begin{proof}  
Clearly, for large prime $q$, we have 
\[ 2 \leq   k <  h  = \lfloor (1+\epsilon)k\rfloor  \leq (1+\epsilon)q^{\epsilon_1} < q.\]
By Propositions \ref{estimate1}, for all large prime $q$, we deduce
\begin{align}
N_h(q) &\geq q^{h-k+1} - \frac{1}{2}(2k)^{h} q^{\frac{h-k+2}{2}}  \nonumber \\
&\geq \frac{1}{2} q^{h-k+1} +  \frac{1}{2}(q^{h-k+1} - (2k)^{h} q^{\frac{h-k+2}{2}}) \nonumber \\
&\geq \frac{1}{2} q^{\epsilon k} +  \frac{1}{2}(q^{\epsilon k } - q^{\epsilon_1(1+\epsilon)k +\frac{\epsilon k+2}{2}}). \nonumber                  
\end{align}
Now, for large prime $q$,  
\[\epsilon k  - (\epsilon_1(1+\epsilon)k +\frac{\epsilon k+2}{2}) = (\frac{\epsilon}{2} - \epsilon_1(1+\epsilon))k - 1 > 0.
 \]
The theorem is proved. 
\end{proof}

\begin{Rmk} When $\epsilon>1$, the dimension $h-k+1 = \lfloor \epsilon k +1\rfloor$ is large compared to $k$ (the number of equations), and non-trivial estimate for $N_h(q)$ 
could be proved using the Weil bound or the Lang-Weil bound. In our applications, we need $\epsilon<1$ and in fact we want $\epsilon$ going to zero. 
In this case, the 
dimension $h-k+1 = \lfloor \epsilon k +1\rfloor$ is very small compared to $k$ (the number of equations), and the estimate in Theorem \ref{weakloc} is much deeper. 
One could ask if there is an elementary 
proof of Theorem \ref{weakloc} bypassing Deligne's theorem for $0<\epsilon < 1$.   

\end{Rmk} 
   
 \section{Inclusion-Exclusion Sieving}
 
To handle the number $N_h^*(q)$ of solutions with distinct coordinates, namely, the number of $\mathbb{F}_q$-rational points on $X_{k,h, \mathbf{u}}^*$, we need to remove those $\mathbb{F}_q$-rational points on $X_{k,h, \mathbf{u}}$ with two coordinates being the same. For this purpose, we introduce the following hyperplane section. 

\begin{Def} 
Let 
\[ Y_{k,h, \mathbf{u}} = X_{k,h, \mathbf{u}} \cap \{x_1 = x_2\}. \]
\end{Def}

Since the defining equations for $X_{k,h, \mathbf{u}}$ are symmetric polynomials in $\{x_1, \dots, x_h\}$, it follows that for any $1\leq i < j \leq h$, we have 
\[ |Y_{k,h, \mathbf{u}}(\mathbb{F}_q)| = |X_{k,h, \mathbf{u}} \cap \{x_1 = x_2\}(\mathbb{F}_q)| = |X_{k,h, \mathbf{u}} \cap \{x_i = x_j\}(\mathbb{F}_q)|. \]
Clearly, 
\[ X_{k,h, \mathbf{u}}^*(\mathbb{F}_q) = X_{k,h, \mathbf{u}}(\mathbb{F}_q) \setminus \bigcup_{1\leq i < j \leq h} \left(X_{k,h, \mathbf{u}} \cap \{x_i = x_j\}(\mathbb{F}_q)\right). \] 
Applying the inclusion-exclusion sieving as done in \cite{CW04}, we obtain the inequality 
\begin{equation}\label{eq:inclusion_exclusion}
N_h^*(q) = |X_{k,h, \mathbf{u}}^*(\mathbb{F}_q)| \geq N_h(q) - \binom{h}{2} |Y_{k,h, \mathbf{u}}(\mathbb{F}_q)|.
\end{equation}
This is a crude sieving, but sufficient for our current purpose. Presumably, the full distinct coordinate sieving developed in \cite{LW10} could be applied to yield a sharp asymptotic formula for $N_h^*(q)$. We do not pursue this refinement here. 

The quantity $N_h(q)$ is already estimated in Proposition \ref{estimate1}. The quantity $|Y_{k,h, \mathbf{u}}(\mathbb{F}_q)|$ can be estimated in a similar manner. 

\begin{Prop}\label{estimate2} 
Let $2\leq k\leq h-2 < q-2$ with $q$ being a prime. Then, 
\[
\left| |Y_{k,h, \mathbf{u}}(\mathbb{F}_q)| - q^{h-k}\right| \leq \frac{1}{2}(2k)^{h-1}q^{\frac{h-k+2}{2}}.
\]
\end{Prop}

\begin{proof}
Note that   
\[ Y_{k,h, \mathbf{u}} = \left(\overline{X}_{k,h, \mathbf{u}} \cap \{x_1 = x_2\}\right) \setminus \left(\overline{X}_{k,h} \cap \{x_1 = x_2\}\right). \]
The projective variety $\overline{X}_{k,h, \mathbf{u}} \cap \{x_1 = x_2\}$ in $\mathbb{P}^{h-1} : = \mathbb{P}^h \cap \{x_1 = x_2\}$ is defined by $k-1$ homogeneous equations. It follows that   
\[ \dim\left(\overline{X}_{k,h, \mathbf{u}} \cap \{x_1 = x_2\}\right) \geq h-1-(k-1) = h-k. \]
By Proposition \ref{comp}, its hyperplane section 
\[ \overline{X}_{k,h, \mathbf{u}} \cap \{x_1 = x_2\} \cap \{ x_2=0\} \cong \overline{X}_{k,h-2, \mathbf{u}} \]
is a complete intersection of dimension $h-k-1$ in $\mathbb{P}^{h-2}$. It follows that   
\[ \dim\left(\overline{X}_{k,h, \mathbf{u}} \cap \{x_1 = x_2\}\right) \leq (h-k-1)+1 = h-k. \]
We thus deduce that   
\[ \dim\left(\overline{X}_{k,h, \mathbf{u}} \cap \{x_1 = x_2\}\right) = h-k. \]  
This proves that $\overline{X}_{k,h, \mathbf{u}} \cap \{x_1 = x_2\}$ is a projective complete intersection of dimension $h-k\geq 2$ in $\mathbb{P}^{h-1}$. By Proposition \ref{smooth}, the singular locus of $\overline{X}_{k,h, \mathbf{u}}$ has dimension bounded by $0$. By Lemma \ref{singular}, the singular locus of $\overline{X}_{k,h, \mathbf{u}} \cap \{x_1 = x_2\}$ has dimension bounded by $1$. Similarly, $\overline{X}_{k,h} \cap \{x_1 = x_2\}$ is a projective complete intersection of dimension $h-k-1\geq 1$ in $\mathbb{P}^{h-1}$ whose singular locus has dimension bounded by $0$ (isolated singularities).  
 
Proposition \ref{Del} with $s=1$ yields the estimate 
\[ \left| |(\overline{X}_{k,h, \mathbf{u}}\cap \{x_1 = x_2\})(\mathbb{F}_{q^e})| - \frac{q^{e(h-k+1)}-1}{q^e-1}\right| \leq C_3 q^{\frac{e(h-k+2)}{2}}, \]
where the constant $C_3$ is independent of the extension degree $e$. Similarly, Proposition \ref{Del} with $s=0$ yields the estimate  
\[ \left| |(\overline{X}_{k,h}\cap \{x_1 = x_2\})(\mathbb{F}_{q^e})| - \frac{q^{e(h-k)}-1}{q^e-1}\right| \leq C_4 q^{\frac{e(h-k)}{2}}, \]
where the constant $C_4$ is independent of the extension degree $e$. Taking the difference of these two asymptotic formulas gives the estimate 
\[ \left| |Y_{k, h, \mathbf{u}}(\mathbb{F}_{q^e})| - q^{e(h-k)}\right| \leq C_5 q^{\frac{e(h-k+2)}{2}}, \] 
where $C_5$ is the total degree of the zeta function of $Y_{k, h, \mathbf{u}}$. Since the affine variety $Y_{k, h, \mathbf{u}}$ in $\mathbb{A}^{h-1}$ is a complete intersection defined by $k-1$ equations of degrees at most $k-1$, the total Betti number estimate in Proposition \ref{betti} yields the estimate  
\[ C_5 \leq \binom{h-2}{k-2} k^{h-1} \leq 2^{h-2} k^{h-1} = \frac{1}{2}(2k)^{h-1}. \]
The proof is complete. 
\end{proof}
  
Now, we are ready to prove the following theorem on locally dense lattices in the weaker sense. 

\begin{Thm}\label{loc} 
Let $0 < \epsilon < 1$. Choose a positive constant $\epsilon_1$ and positive integers $k, h$ such that 
\[ 0 < \epsilon_1 < \frac{\epsilon}{2(1+\epsilon)} < \frac{1}{4}, \quad 2k = 2\lfloor \frac{q^{\epsilon_1}}{2}\rfloor, \quad h = \lfloor (1+\epsilon)k\rfloor \leq  q/2. \]
For $1 \leq p < \infty$ and $\alpha = \left(\frac{1+\epsilon}{2}\right)^{1/p}$, the Reed-Solomon lattice $\mathcal{L}_{q,k}$ in $\mathbb{R}^q$ satisfies the following two properties:
\begin{enumerate}
\item[\rm (1)] $\lambda^{(p)}(\mathcal{L}_{q,k}) \geq (2k)^{1/p}$, and
\item[\rm (2)] Let ${\mathbf{y}} = (1, \dots, 1, 0, \dots, 0)^T$ be the binary vector in $\mathbb{Z}^q$ with Hamming weight $h$. Let $S_2(\mathbf{y}, h)$ be the set of binary vectors in the coset $\mathbf{y} + \mathcal{L}_{q,k}$ with Hamming weight $h$. Then, for all sufficiently large prime $q$, we have 
\[ |S_2(\mathbf{y}, h)| \geq q^{\epsilon_1(1+\epsilon)k}. \]
\end{enumerate}
\end{Thm}

\begin{proof}  
Since $k \leq h \leq q/2$, part (1) holds by Lemma \ref{lem:BP-min-dist}. We can assume that $2 \leq k \leq h-2 < q-2$ with $q$ being a prime, so that the previous estimates on rational points can be applied. Combining Propositions \ref{estimate1} and \ref{estimate2} with the inequality \eqref{eq:inclusion_exclusion}, we obtain 
\begin{align}\label{ineq1}
N_h^*(q) &\geq q^{h-k+1} - \frac{1}{2}(2k)^{h} q^{\frac{h-k+2}{2}} - \binom{h}{2}\left(q^{h-k} + \frac{1}{2}(2k)^{h-1} q^{\frac{h-k+2}{2}}\right) \nonumber \\
&\geq q^{h-k}\left(q - \binom{h}{2}\right) - \binom{h}{2}(2k)^{h}q^{\frac{h-k+2}{2}} \\
&= q^{\frac{h-k+2}{2}} \left(q^{\frac{h-k-2}{2}}\left(q - \binom{h}{2}\right) - \binom{h}{2}(2k)^{h}\right). \nonumber                  
\end{align}
By our assumption, 
\[ h = \lfloor (1+\epsilon)k \rfloor  < 2k \leq q^{\epsilon_1} < q^{1/4}< \sqrt{q}, \quad h-k = \lfloor \epsilon k \rfloor. \]
Then, $q > h^2$ and $q - \binom{h}{2} > \binom{h}{2}$. It follows that 
\begin{equation}\label{ineq2}
N_h^*(q) > q^{\frac{\lfloor \epsilon k\rfloor+2}{2}}\binom{h}{2} \left(q^{\frac{\lfloor \epsilon k\rfloor -2}{2}} - (2k)^{\lfloor (1+\epsilon)k\rfloor}\right).
\end{equation}
By our assumption, 
\[ 2k = 2\lfloor \frac{q^{\epsilon_1}}{2}\rfloor, \quad 0 < \epsilon_1 < \frac{\epsilon}{2(1+\epsilon)} < \frac{1}{4}. \] 
Then, for sufficiently large $q$, one has 
\[ \frac{\lfloor \epsilon k\rfloor -2}{2} - \epsilon_1 \lfloor (1+\epsilon)k\rfloor \geq \left( \frac{\epsilon}{2} - \epsilon_1(1+\epsilon)\right)k - 2 \geq 2. \]
It follows that for large $q$, 
\begin{align*}
\frac{1}{2}q^{\frac{\lfloor \epsilon k\rfloor-2}{2}} - (2k)^{\lfloor (1+\epsilon)k\rfloor} &\geq q^{\epsilon_1 \lfloor (1+\epsilon)k\rfloor}\left(\frac{1}{2}q^{\frac{\lfloor \epsilon k\rfloor-2}{2} - \epsilon_1 \lfloor (1+\epsilon)k\rfloor} - 1\right) \\
&\geq q^{\epsilon_1 \lfloor (1+\epsilon)k\rfloor }\left(\frac{1}{2}q^2 - 1\right) > 0.
\end{align*}
By \eqref{ineq2}, for large $q$, we have  
\[ N_h^*(q) \geq q^{\frac{\lfloor \epsilon k \rfloor+2}{2}}\binom{h}{2}\frac{1}{2}q^{\frac{\lfloor \epsilon k \rfloor-2}{2}} \geq q^{\lfloor \epsilon k\rfloor}. \]
Note that 
\[ \epsilon - \epsilon_1(1+\epsilon) > 2\epsilon_1(1+\epsilon) - \epsilon_1(1+\epsilon) = \epsilon_1(1+\epsilon). \]
Let $\epsilon_2 = \epsilon - 2\epsilon_1(1+\epsilon) > 0$. Since $h = \lfloor (1+\epsilon)k \rfloor \leq 2k$, we conclude that for large $q$, 
\begin{align*}
|S_2(\mathbf{y}, h)| &= \frac{1}{h!}N_h^*(q) \geq \frac{q^{\lfloor \epsilon k\rfloor}}{(2k)^h} \geq q^{\lfloor \epsilon k \rfloor - \epsilon_1 h} \\
&\geq q^{(\epsilon - \epsilon_1(1+\epsilon))k-2} \geq q^{\epsilon_1(1+\epsilon)k + \epsilon_2k - 2} \geq q^{\epsilon_1(1+\epsilon)k}.
\end{align*}
This is subexponential in $q$, as $q$ is a polynomial in $k$.  
\end{proof}

This theorem proves that the explicitly constructed lattice $\mathcal{L}_{q,k}$ in $\mathbb{R}^q$ with the explicit center $\mathbf{y}$ is the desired locally dense lattice we seek, except that we have not considered the linear projection yet. To finish the full proof, we need to show that the projection of $S_2(\mathbf{y}, h)$ to the first $r$ coordinates is surjective onto $\{0, 1\}^r$ for suitable choices of $r$. This is completed in the next section. 

The above theorem already has an interesting application in coding theory. By Theorem \ref{loc} with $p=1$ and \cite[Lemma 5.1]{BP23}, we obtain the following corollary.

\begin{Cor}\label{ratio}  
Let $0 < \epsilon < 1$. Choose a positive constant $\epsilon_1$ and positive integers $k, h$ such that 
\[ 0 < \epsilon_1 < \frac{\epsilon}{2(1+\epsilon)} < \frac{1}{4}, \quad 2k = 2\lfloor \frac{q^{\epsilon_1}}{2}\rfloor, \quad h = \lfloor (1+\epsilon)k\rfloor. \]
For a vector $\mathbf{v} \in \mathbb{F}_q^q$, let $M_q(\mathbf{v}, k, h)$ denote the number of codewords in the Reed-Solomon code $RS_q(h-k+1)$ that each agree with $\mathbf{v}$ in at least $h$ coordinates, i.e.,
\[ M_q(\mathbf{v}, k, h) = \bigl| \{ \mathbf{w} \in RS_q(h-k+1) : d(\mathbf{v}, \mathbf{w}) \leq q-h \}\bigr|, \]
where $d(\mathbf{v}, \mathbf{w})$ denotes the Hamming distance.   
Then, for all sufficiently large prime $q$, there is an explicit center $\mathbf{v} \in \mathbb{F}_q^q$ satisfying   
\[ M_q(\mathbf{v}, k, h) \geq q^{\epsilon_1(1+\epsilon)k}, \]
which shows that $M_q(\mathbf{v}, k, h)$ is at least subexponential in $q$. 
\end{Cor}

\begin{proof}
Recall that $\mathbb{F}_q = \{ a_1, \dots, a_q \}$ with a fixed ordering. For every binary vector $\mathbf{x} = (x_1, \dots, x_q)^T \in S_2(\mathbf{y}, h)$, define the monic squarefree polynomial 
\[ f_{\mathbf{x}}(t) = \prod_{i=1}^q (t - a_i)^{x_i} \in \mathbb{F}_q[t], \]
of degree $h$ with $h$ distinct roots in $\mathbb{F}_q$. Since $H_q(k)\mathbf{x} = H_q(k)\mathbf{y}$, the $j$-th power symmetric function of the roots of $f_{\mathbf{x}}(t)$ is equal to the $j$-th power symmetric function of the roots of $f_{\mathbf{y}}(t)$ for every $1 \leq j \leq k-1$. By Newton's identities, the $j$-th elementary symmetric function of the roots of $f_{\mathbf{x}}(t)$ is equal to the $j$-th elementary symmetric function of the roots of $f_{\mathbf{y}}(t)$ for every $1 \leq j \leq k-1$.   

This means that the coefficient of $t^{h-j}$ in $f_{\mathbf{x}}(t)$ is equal to the coefficient of $t^{h-j}$ in $f_{\mathbf{y}}(t)$ for every $0 \leq j \leq k-1$. It follows that $f_{\mathbf{y}}(t) - f_{\mathbf{x}}(t) \in \mathbb{F}_q[t]$ is a polynomial of degree at most $h-k < q$, which uniquely represents a codeword in the Reed-Solomon code $RS_q(h-k+1)$ via evaluation at the set $\mathbb{F}_q$. Distinct choices of $\mathbf{x}$ yield distinct polynomials $f_{\mathbf{x}}(t)$, and hence a distinct codeword $f_{\mathbf{y}}(t) - f_{\mathbf{x}}(t)$ in the Reed-Solomon code $RS_q(h-k+1)$. 

The explicit center $\mathbf{v}$ is the word represented by $f_{\mathbf{y}}(t)$. It clearly agrees with the codeword  $f_{\mathbf{y}}(t) - f_{\mathbf{x}}(t)$ at all the $h$ roots of $f_{\mathbf{x}}(t)$ for every $x \in S_2(\mathbf{y}, h)$. Thus, their Hamming distance is at most $q-h$. By Theorem \ref{loc}, we conclude that 
\[ M_q(\mathbf{v}, k, h) \geq |S_2(\mathbf{y}, h)| > q^{\epsilon_1(1+\epsilon)k}. \]
\end{proof}

\begin{Rmk} 
In Corollary \ref{ratio}, the radius $q-h$ is smaller than the minimum distance $q+k-h$ of the Reed-Solomon code $RS_q(h-k+1)$, and the agreement-to-dimension ratio is at least 
\[ \frac{h}{h-k+1} = \frac{\lfloor (1+\epsilon)k\rfloor}{\lfloor \epsilon k\rfloor + 1} \sim \frac{1+\epsilon}{\epsilon}, \]
as $q$ grows. This limit can be made as large as one wishes as $\epsilon$ approaches zero. In contrast, the state of the art for explicit Reed-Solomon list-decoding configurations requires a ratio of $2 - \Omega(1)$ in order to obtain a super-polynomial list size; see \cite[Corollary 2]{GR05}.   
\end{Rmk}

\section{Linear Projections} 

In this final section, we assume that we are in the situation of Theorem \ref{loc}. To finish the full de-randomization, we now consider the linear projection to the first $r$ coordinates: 
\[ A : (x_1, \dots, x_q)^T \in S_2(\mathbf{y}, h) \to (x_1, \dots, x_r)^T \in \{0,1\}^r. \]
Note that the map $A$ is now restricted to the finite subset $S_2(\mathbf{y}, h)$ of $\mathbb{Z}^q$. We need to show that this projection map is surjective for 
\[ r = \lfloor q^{\delta} \rfloor \leq h - k -2 = \lfloor \epsilon k\rfloor -2 < \lfloor \epsilon k\rfloor \]
for some constant $0 < \delta < 1$ and all sufficiently large primes $q$. Then, its image contains $2^r = 2^{\lfloor q^{\delta} \rfloor}$ elements, which is subexponential in $q$. The proof of surjectivity is conceptually similar to the proof of Theorem \ref{loc}, but notationally and combinatorially more involved. We present the details below. 

Fix $x = (x_1, \dots, x_r)^T \in \{0,1\}^r$. Let $t$ be the Hamming weight of $x$, i.e., the number of non-zero entries in $x$. Thus, $0 \leq t \leq r$. The fiber $A^{-1}(x_1, \dots, x_r)^T$ in $S_2(\mathbf{y}, h)$ consists of all binary vectors $(x_1, \dots, x_r, x_{r+1}, \dots, x_q)^T$ of Hamming weight $h$ satisfying the following system of equations:
\[ \sum_{i=1}^q a_i^j x_i = h_j, \quad 1 \leq j \leq k-1. \]
Since the first $r$ coordinates $(x_1, \dots, x_r)$ are fixed, the fiber $A^{-1}(x_1, \dots, x_r)^T$ in $S_2(\mathbf{y}, h)$ can be identified with its last $q-r$ coordinates, which consist of a binary vector $(x_{r+1}, \dots, x_q)^T$ of Hamming weight $h-t$ satisfying:
\[ \sum_{i=r+1}^q a_i^j x_i = h_j - \sum_{i=1}^r a_i^j x_i, \quad 1 \leq j \leq k-1. \]
Let the non-zero coordinates in $(x_{r+1}, \dots, x_q)$ be $x_{e_1} = \dots = x_{e_{h-t}} = 1$, where 
\[ r+1 \leq e_1 < \dots < e_{h-t} \leq q. \] 
Then, we have   
\[ \sum_{i=1}^{h-t} a_{e_i}^j = h_j - \sum_{i=1}^r a_i^j x_i, \quad 1 \leq j \leq k-1. \]

\begin{Def} 
Fix $x = (x_1, \dots, x_r)^T \in \{0,1\}^r$ with Hamming weight $t$. Let $Z_x$ be the affine variety in $\mathbb{A}^{h-t}$ defined by the system 
\[ \sum_{i=1}^{h-t} z_i^j = h_j - \sum_{i=1}^r a_i^j x_i, \quad 1 \leq j \leq k-1. \]
Let $Z_x^*$ denote the subvariety of $Z_x$ consisting of all points with distinct coordinates. That is, 
\[ Z_x^* : =  Z_x \setminus \bigcup_{1 \leq i_1 < i_2 \leq h-t} \{ z_{i_1} - z_{i_2} = 0 \}. \]
Let $Z_x^{**}$ be the subvariety of $Z_x^*$ defined by requiring $z_i \notin \{a_1, \dots, a_r\}$ for all $1 \leq i \leq h-t$. That is,   
\[ Z_x^{**} : = Z_x^* \setminus \bigcup_{\substack{1 \leq i \leq h-t \\ 1 \leq e \leq r}} \{ z_i = a_e \}. \]
\end{Def}

Recall that $\mathbb{F}_q = \{a_1, \dots, a_q\}$ with a fixed ordering. For any $\mathbb{F}_q$-rational point $(z_1, \dots, z_{h-t})$ on $Z_x^{**}$ with distinct coordinates, since $z_i \notin \{a_1, \dots, a_r\}$, we can write uniquely 
\[ z_1 = a_{e_1}, \dots, z_{h-t} = a_{e_{h-t}}, \]
where $\{ e_1, \dots, e_{h-t} \}$ are distinct elements of $\{ r+1, \dots, q \}$. For $r+1 \leq j \leq q$, define $x_j = 1$ if $j = e_i$ for some $i$, and $x_j = 0$ otherwise. Then, 
\[ (x_1, \dots, x_r, x_{r+1}, \dots, x_q)^T \in S_2(\mathbf{y}, h) \]
with the projection property that 
\[ A (x_1, \dots, x_r, x_{r+1}, \dots, x_q)^T = (x_1, \dots, x_r)^T \in \{0, 1\}^r. \] 
Two $\mathbb{F}_q$-rational points $(z_1, \dots, z_{h-t})$ and $(z'_1, \dots, z'_{h-t})$ on $Z_x^{**}$ yield the same element in $S_2(\mathbf{y}, h)$ if and only if as sets, 
\[ \{z_1, \dots, z_{h-t}\} = \{z'_1, \dots, z'_{h-t}\}. \]
This occurs if and only if $(z_1, \dots, z_{h-t})$ and $(z'_1, \dots, z'_{h-t})$ are permutations of each other. This establishes the following proposition:
  
\begin{Prop} 
For $x = (x_1, \dots, x_r)^T \in \{0,1\}^r$ with $t$ being the number of non-zero entries in $x$, we have 
\begin{equation}\label{eq:fiber_bound}
|A^{-1}(x_1, \dots, x_r)^T| \geq \frac{1}{(h-t)!} |Z_x^{**}(\mathbb{F}_q)|.
\end{equation}
\end{Prop}

We need to show that the left-hand side of \eqref{eq:fiber_bound} is strictly positive. For this, it suffices to show that the number $|Z_x^{**}(\mathbb{F}_q)|$ of $\mathbb{F}_q$-rational points on the variety $Z_x^{**}$ is positive. In analogy to Theorem \ref{loc}, we will demonstrate that $|Z_x^{**}(\mathbb{F}_q)|$ is subexponentially large in $q$ under suitable parameter selections. 
 
For $1 \leq i \leq h-t$ and $1 \leq e \leq r$, let 
\[ Z_{x, i, e}^* : = Z_x^* \cap \{ z_i = a_e \}, \quad Z_{x, i, e} : =  Z_x \cap \{ z_i = a_e \}. \]
By inclusion-exclusion sieving, we have 
\begin{equation}\label{eq:ineq_star}
|Z_x^{**}(\mathbb{F}_q)| \geq |Z_x^*(\mathbb{F}_q)| - \sum_{i=1}^{h-t} \sum_{e=1}^r |Z_{x, i, e}^*(\mathbb{F}_q)|.
\end{equation}
The variety $Z_x^*$ is structurally identical to the variety $X_{k, h-t, \mathbf{u}}^*$ in the affine space $\mathbb{A}^{h-t}$ defined previously, yielding $\dim(Z_x^*) = h-t-k+1$. Similarly, $Z_{x, i, e}^*$ is an open subset of $X_{k, h-1-t, \mathbf{u}}^*$ in the affine space $\mathbb{A}^{h-1-t}$, and $Z_{x, i, e}$ is isomorphic to $X_{k, h-1-t, \mathbf{u}}$ in $\mathbb{A}^{h-1-t}$ with $\dim(Z_{x, i, e}) = h-t-k$.

By our choice of $r$ with $r < h-k$, we have $k < h - r \leq h - t < q - 2$. Applying inequality \eqref{ineq1} and Proposition \ref{estimate1}, we deduce 
\begin{equation}
|Z_x^*(\mathbb{F}_q)| > q^{h-t-k} \left(q - \binom{h-t}{2}\right) - \binom{h-t}{2}(2k)^{h-t} q^{\frac{h-t-k+2}{2}},
\end{equation}
\begin{equation}
|Z_{x, i, e}^*(\mathbb{F}_q)| \leq |Z_{x, i, e}(\mathbb{F}_q)| \leq q^{h-t-k} + (2k)^{h-1-t} q^{\frac{h-t-k+1}{2}}.
\end{equation}
These bounds together with inequality \eqref{eq:ineq_star} imply 
\[ |Z_x^{**}(\mathbb{F}_q)| \geq q^{h-t-k} \left(q - \binom{h}{2} - rh\right) - \left(\binom{h}{2} + rh\right)(2k)^{h-t} q^{\frac{h-t-k+2}{2}}. \]
Now, select 
\begin{equation}\label{eq3}
\quad 0 < \epsilon_1 < \frac{\epsilon}{2(1+\epsilon)} < \frac{1}{4}, ~ 2k = 2\lfloor \frac{q^{\epsilon_1}}{2}\rfloor,  ~h = \lfloor (1+\epsilon)k \rfloor, \quad r < h-k \leq \frac{h-2}{2}. 
\end{equation}
The last inequality holds for sufficiently large $q$ because $\epsilon < \frac{1+\epsilon}{2}$ for $0 < \epsilon < 1$. 
Furthermore, 
\[ h < 2k \leq q^{\epsilon_1} \leq q^{1/4} <  \frac{1}{2}\sqrt{q}.\]
This implies 
\[ q > 4h^2 > 4\binom{h}{2} \geq 2\left(\binom{h}{2} + rh\right). \]
We thus obtain 
\begin{align*}
|Z_x^{**}(\mathbb{F}_q)| &\geq \left(\binom{h}{2} + rh\right) \left( q^{h-t-k} - (2k)^{h-t}q^{\frac{h-t-k+2}{2}}\right) \\
&= \left(\binom{h}{2} + rh\right) q^{\frac{\lfloor \epsilon k\rfloor-t+2}{2}} \left( q^{\frac{\lfloor \epsilon k\rfloor-t-2}{2}} - (2k)^{\lfloor(1+\epsilon)k\rfloor-t}\right).
\end{align*}
Choosing 
\begin{equation}\label{eq4}
r = \lfloor q^{\delta} \rfloor < \min\left(2 \left( \frac{\epsilon}{2} - \epsilon_1(1+\epsilon)\right)k - 5, \, \frac{\epsilon_1(1+\epsilon)}{4}k - 1\right).
\end{equation}
In particular, the rightmost term guarantees that 
\[ r < \frac{\epsilon_1(1+\epsilon)}{4}k - 1 < \frac{\epsilon}{8} k < \lfloor \epsilon k \rfloor = h-k. \]
Such a choice of $r$ with $0 < \delta < \epsilon_1$ is clearly viable for sufficiently large $q$. Recall that $0 \leq t \leq r$. We deduce 
\begin{align*}
 \frac{\lfloor \epsilon k \rfloor - t - 2}{2} - \epsilon_1 (\lfloor(1+\epsilon)k\rfloor - t) &\geq \left( \frac{\epsilon}{2} - \epsilon_1(1+\epsilon)\right)k - \frac{r+3}{2} \\ 
 &\geq \frac{r+5}{2} - \frac{r+3}{2} =1. 
\end{align*}
It follows that for large $q$, 
\begin{align*}
\frac{1}{2}q^{\frac{\lfloor \epsilon k\rfloor -t-2}{2}} - (2k)^{\lfloor (1+\epsilon)k\rfloor-t} &\geq q^{\epsilon_1 (\lfloor (1+\epsilon)k\rfloor - t)} \left( \frac{1}{2}q^{\frac{\lfloor \epsilon k \rfloor -t-2}{2} - \epsilon_1 (\lfloor (1+\epsilon)k\rfloor-t)} - 1\right) \\
&\geq q^{\epsilon_1 (\lfloor (1+\epsilon)k\rfloor - r)} \left(\frac{1}{2}q - 1\right) \geq q^{\epsilon_1 (\lfloor(1+\epsilon)k\rfloor - \lfloor \epsilon k\rfloor)} \geq 1.
\end{align*}
Thus, for sufficiently large $q$, 
\[ |Z_x^{**}(\mathbb{F}_q)| \geq q^{\frac{\lfloor \epsilon k \rfloor-t+2}{2}} \left(\binom{h}{2} + rh\right) \frac{1}{2} q^{\frac{\lfloor \epsilon k\rfloor-t-2}{2}} \geq q^{\lfloor \epsilon k\rfloor - t} \geq q^{\epsilon k - r - 1}. \]
Since $h-t \leq h \leq 2k$, $2k \leq q^{\epsilon_1}$, and $r < \frac{\epsilon_1(1+\epsilon)}{4}k - 1$, we conclude that for large $q$, 
\begin{align*}
|A^{-1}(x_1, \dots, x_r)^T| &\geq \frac{1}{(h-t)!} |Z_x^{**}(\mathbb{F}_q)| \geq \frac{q^{\epsilon k - r - 1}}{(2k)^h} \\
&\geq q^{\epsilon k - r - 1 - \epsilon_1 h} \geq q^{(\epsilon - \epsilon_1(1+\epsilon))k - r - 1} \\
&\geq q^{\epsilon_1(1+\epsilon)k - r - 1} \geq q^{\epsilon_1(1+\epsilon)k - \frac{\epsilon_1(1+\epsilon)}{4}k} \\
&= q^{\frac{3}{4}\epsilon_1(1+\epsilon)k}.
\end{align*}
This is subexponential in $q$, since $q$ is polynomial in $k$. In particular, the projection map $A : S_2(\mathbf{y}, h) \to \{ 0, 1\}^r$ is surjective with the above choices of $h = \lfloor(1+\epsilon)k\rfloor$, $r = \lfloor q^{\delta}\rfloor$, and $2k = 2\lfloor \frac{q^{\epsilon_1}}{2}\rfloor$ as defined in \eqref{eq3} and \eqref{eq4}. 

In summary, we have established the following theorem on locally dense lattices:

\begin{Thm}\label{thm-final} 
 Let $0 < \epsilon < 1$ be a constant. Choose positive constants  $\delta$ and $\epsilon_1$ such that 
\[ 0 < \delta < \epsilon_1 < \frac{\epsilon}{2(1+\epsilon)} < \frac{1}{4}.\]
For all sufficiently large prime $q$, define positive integers $k, h, r$ by 
\[\quad 2k = 2\lfloor \frac{q^{\epsilon_1}}{2}\rfloor, \quad h = \lfloor (1+\epsilon)k\rfloor,  \quad r = \lfloor q^{\delta} \rfloor.\]
For $1 \leq p < \infty$ and $\alpha = \left(\frac{1+\epsilon}{2}\right)^{1/p}$, the Reed-Solomon lattice $\mathcal{L}_{q,k}$ in $\mathbb{R}^q$ is a $(p, \alpha, r, q)$-locally dense lattice satisfying the following two properties:
\begin{enumerate}
\item[\rm (1)] $\lambda^{(p)}(\mathcal{L}) \geq (2k)^{1/p}$, and
\item[\rm (2)] Let $\mathbf{y} = (1, \dots, 1, 0, \dots, 0)^T$ be the binary vector in $\mathbb{Z}^q$ with Hamming weight $h$. Let $V : =  (\mathbf{y} + \mathcal{L}_{q,k}) \cap B(0, \alpha (2k)^{1/p})$ be the set of all vectors in the coset $\mathbf{y} + \mathcal{L}_{q, k}$ with $\ell_p$-norm at most $\alpha (2k)^{1/p}$. Then, 
\[ \{ 0, 1\}^r \subseteq A(V) : = \{(x_1, \dots, x_r) : (x_1, \dots, x_q) \in V\}, \]
where $A$ is the projection to the first $r$ coordinates. 
\end{enumerate}
\end{Thm}

Note that the condition $\delta < \epsilon_1$ automatically implies that the following condition 
\[ r = \lfloor q^{\delta} \rfloor < \min\left(2 \left( \frac{\epsilon}{2} - \epsilon_1(1+\epsilon)\right)k - 5, \, \frac{\epsilon_1(1+\epsilon)}{4}k - 1\right) \]
is satisfied for all sufficiently large $q$.  This is the condition in (\ref{eq4}). 
The locally dense lattice $\mathcal{L}_{q,k}$ in $\mathbb{R}^q$ is constructed deterministically in $\mathrm{poly}(q) = \mathrm{poly}(r)$ time. 

We now apply Theorem \ref{thm:reduction} to finish the proof of the derandomization. 
Let $1\leq p < \infty$. Let $r_0$ be a sufficiently large integer and $\alpha \in (2^{-1/p}, 1)$. Set 
$\epsilon = 2\alpha^p -1 >0$ and choose positive constants $\delta$ and $\epsilon_1$ as in Theorem \ref{thm-final}. Choose a poly($r_0$)-bounded prime $q > r_0^{1/\delta}$, i.e., $q^{\delta} > r_0$.  
Such a prime $q$ exists by Bertrand's theorem and can be constructed in deterministic poly($r_0$) time 
by scanning through the polynomial-size interval and using a deterministic polynomial time primality test algorithm. Let 
\[r = \lfloor q^{\delta}\rfloor, \quad 2k = 2\lfloor \frac{q^{\epsilon_1}}{2}\rfloor, \quad h = \lfloor (1+\epsilon)k\rfloor.\] 
It is clear that $r\geq r_0$. Theorem \ref{thm-final} shows that the 
lattice $\mathcal{L}_{q,k}$ in $\mathbb{R}^q$ is a $(p, \alpha, r, q)$-locally dense lattice, and it is constructed in deterministic poly($r$)-time. 
The surjective projection  $\{ 0, 1\}^r$ to the first $r_0$-coordinates $\{ 0, 1\}^{r_0}$ shows that 
$\mathcal{L}_{q,k}$ in $\mathbb{R}^q$ is also a $(p, \alpha, r_0, q)$-locally dense lattice constructed in deterministic poly($r_0$)-time. 
Theorem \ref{thm:reduction} applies. Since we can take $\alpha$ arbitrarily close to $2^{-1/p}$,  
we deduce the following main hardness result:

\begin{Thm} 
For every $1\leq p <\infty$, $\gamma\text{-}\mathrm{GapSVP}_p$ is $\mathbf{NP}$-hard for all $1 \leq \gamma < 2^{1/p}$. 
\end{Thm}

\end{document}